\documentclass{amsart}
\pdfoutput=1
\usepackage{float}
\restylefloat{figure}
\usepackage{mathtools}
\usepackage{amssymb}
\usepackage{mathrsfs}
\usepackage{bm}
\usepackage[all]{xy}
\UseComputerModernTips
\CompileMatrices
\usepackage[colorlinks,allcolors=blue]{hyperref}
\usepackage{ifthen}
\usepackage{latexsym}
\usepackage{tikz}
\usepackage{tikz-3dplot}
\usetikzlibrary{backgrounds}
\allowdisplaybreaks
\numberwithin{equation}{section}
\newtheorem{theorem}[equation]{Theorem}
\newtheorem{lemma}[equation]{Lemma}
\newtheorem{proposition}[equation]{Proposition}
\newtheorem{corollary}[equation]{Corollary}

\theoremstyle{definition}
\newtheorem{definition}[equation]{Definition}

\newtheorem{context}[equation]{B\'ezout Context}

\theoremstyle{remark}

\renewcommand{\phi}{\varphi}
\newcommand{\bndry}{\partial}
\DeclareMathSymbol{\boxprod}{\mathbin}{AMSa}{"03} 
\DeclareMathSymbol{\mixprod}{\mathbin}{AMSa}{"4F} 

\newcommand{\dirsum}{\oplus}

\newcommand{\disjunion}{\sqcup}

\newcommand{\dual}{^\vee}

\newcommand{\homeo}{\approx}

\newcommand{\includesin}{\hookrightarrow}
\newcommand{\intersect}{\cap}
\newcommand{\iso}{\cong}
\newcommand{\Mackey}[1]{{\underline {#1}}}

\newcommand{\tensor}{\otimes}
\newcommand{\union}{\cup}

\newcommand{\C}{{\mathbb C}}

\newcommand{\NN}{{\mathbb N}}

\newcommand{\PP}{\mathbb{P}}
\newcommand{\R}{{\mathbb R}}

\newcommand{\Z}{\mathbb{Z}}
\newcommand{\ZZ}{\mathbb{Z}}
\newcommand{\GG}{{C_2}}
\newcommand{\cwt}{\zeta_1}
\newcommand{\cxwt}{\zeta_0}
\newcommand{\cwd}{\widehat{c}_\omega}
\newcommand{\cxwd}{\widehat{c}_{\chiw}}
\newcommand{\cd}[1][{}]{\,\widehat{c}^{\;#1}}
\newcommand{\Cpq}[2]{\C^{#1+#2\sigma}}
\newcommand{\Cq}[1]{\C^{#1\sigma}}
\newcommand{\Xpq}[2]{\PP(\Cpq{#1}{#2})}
\newcommand{\partition}[1]{\underline{{#1}}}
\newcommand{\IScl}[1]{Y_{\partition{#1}}}   
\newcommand{\ISc}[3]{Y_{#1}^{#2,#3}}
\DeclareMathOperator{\Gr}{Gr}
\newcommand{\BScl}[1]{\mathbb{Y}_{\partition{#1}}} 
\newcommand{\BSc}[3]{\mathbb{Y}_{#1}^{#2,#3}} 
\newcommand{\tBSc}[3]{\widetilde{\mathbb{Y}}_{#1}^{#2,#3}} 

\newcommand{\Xp}[1]{\PP(\C^{#1})}
\newcommand{\Xq}[1]{\PP(\Cq{#1})}

\newcommand{\chiw}{\chi\omega}

\newcommand{\Fpqk}[3]{V_{#1,#2}^{#3}}                               
\newcommand{\Freep}[1]{\mathbb{F}^{#1}}                             
\newcommand{\Freel}[1]{\mathbb{F}_{#1}}  
\newcommand{\Spqk}[3]{\mathbb{X}_{{#1},{#2}}^{#3}}                  

\newcommand{\tSpqk}[3]{\widetilde{\mathbb{X}}_{{#1},{#2}}^{#3}}     

\newcommand{\eoz}{Q}
\newcommand{\exoz}{\chi Q}
\newcommand{\rem}[1]{\overline{#1}}

\newcommand{\I}{{\mathrm{I}}}
\newcommand{\II}{{\mathrm{II}}}
\newcommand{\III}{{\mathrm{III}}}
\newcommand{\IV}{{\mathrm{IV}}}

\DeclareMathOperator{\grad}{grad}
\newcommand{\copt}{{\mathbb H}}
\begin{document}

\title{A geometric $\GG$-equivariant B\'{e}zout theorem}
\author{Steven R. Costenoble}
\address{Steven R. Costenoble, Department of Mathematics\\Hofstra University\\
   Hempstead, NY 11549}
\email{Steven.R.Costenoble@Hofstra.edu}
\author{Thomas Hudson}
\address{Thomas Hudson, College of Transdisciplinary Studies, DGIST, 
Daegu, 42988, Republic of Korea}
\email{hudson@dgist.ac.kr}

\begin{abstract}
Classically, B\'ezout's theorem says that an intersection of hypersurfaces in a projective space is
rationally equivalent to a number of copies of a smaller projective space, the number depending on
the degrees of the hypersurfaces.
We give a generalization of that result to the context of $C_2$-equivariant hypersurfaces in
$C_2$-equivariant linear projective space, expressing the intersection as a linear combination of
equivariant Schubert varieties.
\end{abstract}

\makeatletter
\@namedef{subjclassname@2020}{%
  \textup{2020} Mathematics Subject Classification}
\makeatother

\maketitle

\tableofcontents

\section*{Introduction}

The general nonequivariant B\'ezout theorem, as given in \cite{IntersectionFulton}, for example,
says the following. Suppose that we have $n$ nonzero homogeneous polynomials $f_i$, $1\leq i\leq n$,
in $N>n$ complex variables. Let $d_i$ be the degree of $f_i$ and let $\Delta = d_1d_2\cdots d_n$.
Each $f_i$ can be thought of as a section of the line bundle $O(d_i)$, the $d_i$-fold tensor power
of the dual of the tautological bundle over $\PP^{N-1}$.
The zero locus of $f_i$ is $H_i$, a hypersurface in $\PP^{N-1}$, and the common zero locus
of all the polynomials is the intersection $\bigcap_i H_i$.
B\'ezout's theorem then says that, in the Chow ring of $\PP^{N-1}$, we have
\[
    \Bigl[\bigcap_i H_i\Bigr] = \Delta[\PP^{N-n-1}] ,
\]
that is, the common zero locus is generically rationally equivalent to $\Delta$ copies of $\PP^{N-n-1}$.
The classical case is when $n = N-1$, in which case the zero locus consists of $\Delta$ points,
assuming the points are counted with multiplicity.
Our main result is Theorem~\ref{thm:Bezout}, which generalizes
B\'ezout's theorem to the simplest equivariant context, considering actions
of the group $C_2 = \{1, t\}$. 
The result involves a number of interesting subspaces of equivariant projective space,
no longer just smaller projective spaces.


The usual approach to proving B\'ezout's theorem nonequivariantly is to first 
notice that $[H_i] = e(O(d_i))$, the Euler class of the line bundle $O(d_i)$,
and that 
\[
    \Bigl[\bigcap_i H_i\Bigr] = e(O(d_1)\dirsum\cdots\dirsum O(d_n)) = \prod_i e(O(d_i)).
\]
The Chow ring is computed to be $\ZZ[\cd]/\langle \cd[N]\rangle$,
with $\cd[n] = [\PP^{N-n-1}]$,
and $[H_i] = e(O(d_i)) = d_i\cd$,
from all of which B\'ezout's theorem easily follows.
In place of the Chow ring we could use singular cohomology $H^*(\PP^{N-1};\ZZ)$, which is isomorphic
to the Chow ring in this case.

To generalize this to our equivariant context, we replace
$\PP^{N-1}$ by the $C_2$-space $\Xpq pq$,
the space of complex lines in $\Cpq pq = \C^p\dirsum\Cq q$, where $\C^p$ has trivial $C_2$-action and
$\Cq q = (\Cq{})^q$ with $\Cq{}$ being $\C$ with $t$ acting as $-1$.
In \cite{CHTAlgebraic} we carried out the algebraic part of this project.
Equivariantly, we do not yet know the proper analogue of the Chow ring,
let alone its value on $\Xpq pq$, so
in its place we use equivariant ordinary cohomology, with the extended grading developed
in \cite{CostenobleWanerBook}, and the calculation of the ordinary cohomology
of complex projective spaces given in \cite{CHTFiniteProjSpace},
where we showed that the equivariant ordinary cohomology
of $\Xpq pq$ is free over the $RO(\GG)$-graded cohomology of a point and
gave an explicit basis.
In \cite{CHTAlgebraic} we computed the Euler classes of sums of line bundles,
expressing them in terms of that basis.

In this paper, we redo the calculation in order to give a geometric interpretation
similar to the nonequivariant B\'ezout theorem.
There are several issues we need to address.
The first is to give geometric meaning to ordinary equivariant (co)homology.
As is well-known nonequivariantly, not all ordinary homology classes of a space $X$ can be represented
by manifolds, that is, as the images of fundamental classes of smooth manifolds mapping to $X$.
However, there is the Baas-Sullivan approach that allows us to consider ordinary homology
as the bordism theory of manifolds with certain kinds of singularities.
Hastings and Waner generalized this approach to $RO(G)$-graded ordinary homology
and their results can be generalized further to the ordinary theory we use here.
This can then be used to interpret the ordinary cohomology of $G$-manifolds.
This interpretation is discussed in \S\ref{sec:singular}.

Another issue is that, nonequivariantly, the integers $\ZZ$ play two coincidental roles:
On the one hand, we have the calculational fact that $H^*(\PP^{N-1};\ZZ)$ is a free
$\ZZ$-module, with basis given by the powers of $\cd$.
The equivariant analogue is that the equivariant ordinary cohomology of $\Xpq pq$
is free over the $RO(\GG)$-graded cohomology of a point
(with a more complicated, but known, basis).
But note that the equivariant cohomology of a point is not simply $\ZZ$.
(See Appendix~\ref{app:ordinarycohomology} for a review of equivariant ordinary cohomology
theory, including the structure of the cohomology of a point.)
On the other hand, because the coefficients are $\ZZ$,
the nonequivariant B\'ezout theorem is often thought of as
a counting result: The common zero locus of the polynomials $f_i$,
represented by the product of Euler classes,
is equivalent to $\Delta$ copies of a projective space.
Given that the equivariant cohomology of a point is not just $\ZZ$,
our calculation in \cite{CHTAlgebraic}, which expressed Euler classes in terms
of a basis, using coefficients from the cohomology of a point, is not suitable for counting in this way.
In \S\ref{sec:EulerSums} we redo the algebraic calculation in a way that will
allow us to express the answer as a linear combination of classes represented by (singular) manifolds,
with integer coefficients,
which will be the form of our geometric B\'ezout theorem.

The singular manifolds that appear in our results are analogues of Schubert varieties.
Nonequivariantly, we can consider Schubert varieties in projective spaces, but
this is not usually done explicitly because
they are just smaller projective spaces.
Equivariantly, things get more interesting.
We believe that the Schubert varieties we introduce here will underlie similar calculations
for equivariant Grassmannians, but that is work that remains to be done. To get a feeling for what kind of equivariant spaces appear in our result, the  reader may want to look at the discussion of the dimension-1 case in \S\ref{sec:Bezout}, where we provide an explicit description of the Schubert varieties involved in the formula.

This paper is organized as follows.
To keep this paper as self-contained as possible, in an appendix we give the facts we need about equivariant ordinary cohomology
with extended grading
and the cohomology of $\Xpq pq$, summarizing results from \cite{CostenobleWanerBook}
and \cite{CHTFiniteProjSpace}.
As mentioned above, in \S\ref{sec:EulerSums} we redo the algebraic calculation of \cite{CHTAlgebraic}, of the Euler class of a sum
of line bundles, in a form better suited to give us our geometric result.
In \S\ref{sec:singular}, we review the approach of Hastings and Waner to representing equivariant ordinary homology
by singular manifolds, and its generalization to ordinary homology with extended grading.
In \S\ref{sec:Schubert} we introduce the singular manifolds used in our B\'ezout theorem,
which are close analogues of nonequivariant Schubert varieties, and
in \S\ref{sec:representing} we discuss how various terms appearing in our calculations
from \S\ref{sec:EulerSums} are represented by these Schubert varieties.
Finally, in \S\ref{sec:Bezout}, we put these results together to state our geometric B\'ezout theorem,
and give examples that show what kinds of singular manifolds appear in the codimension-1 case and the cases of 
dimensions~0, 1, and~2.

\subsection*{Acknowledgments} Both authors would like to thank Sean Tilson for his help in laying the foundations of this work and for bridging between two otherwise unconnectable worlds. The second author was partially supported by the DFG through the SPP 1786: \textit{Homotopy theory and Algebraic Geometry}, Project number 405468058: \textit{$C_2$-equivariant Schubert calculus of homogeneous spaces}.
The initial stages of this research took place while the second author was affiliated to the Bergische Universit\"{a}t Wuppertal, there he benefited from the  the activities of the research training group
\emph{GRK 2240: Algebro-Geometric Methods in Algebra, Arithmetic and Topology},
which is funded by the DFG.

\section{The Euler class of a sum of line bundles}\label{sec:EulerSums}

In \cite[Proposition 6.12]{CHTFiniteProjSpace} we gave a formula for the Euler class of a sum of line bundles over an equivariant projective space $\Xpq{p}{q}$. The goal of this section is to provide an alternative expression for the same quantity that lends itself better to a geometric interpretation. We begin by establishing some notation. Let $F=L_1\oplus\cdots \oplus L_n$ be the direct sum of $n$ line bundle $L_i\rightarrow \Xpq{p}{q}$. Since the Euler class is multiplicative with respect to direct sum, we have $e(F)=\prod_{i=1}^n e(L_i)$. 
In the nonequivariant case this computation is quite simple since one can always find $d_i\in \ZZ$ for which $L_i\iso O(d_i)$ and the additive formal group law in cohomology allows one to conclude that $e(L_i)=d_i e(O(1))$. As a consequence, we get that the nonequivariant Euler class of $F$ is given by $e(F)=d_1\cdots d_n e(O(1))^n$, where the product $d:=d_1\cdots d_n$ is the degree of $e(F)$ viewed as an element of $H^{2n}(\Xp{p+q-1})$. 

In equivariant cohomology we do not have a formal group law at our disposal, but we can still calculate the Euler classes of line bundles. In \cite[Theorem~6.3]{CHTFiniteProjSpace}, we showed that line bundles over equivariant projective spaces can be conveniently subdivided into four families, allowing us to describe their Euler classes in terms of the multiplicative generators $\cxwt$, $\cwt$, $\cwd=e(O(1))$, $\cxwd=e(\chi O(1))$, together with 
\begin{align*}
\eoz &:= e(O(2)) = \tau(c) + e^{-2}\kappa \cwd\cxwd \qquad\text{and} \\
\exoz &:= e(\chi O(2)) = \cxwt\cwd + \cwt\cxwd\,.
\end{align*}
The exact expressions were obtained in \cite[Proposition 6.5]{CHTFiniteProjSpace} and, by making use of $\eoz$ and $\exoz$, can be restated as as follows.

\begin{proposition}\label{prop: base case}
In the cohomology of $\Xpq pq$, for $k\in\Z$, we have:
\begin{alignat*}{2}
 \textup{i)}&\quad& e(O(2k+1)) &= \cwd(1 + k\tau(1) + ke^{-2}\kappa\cwt\cxwd)
  =\cwd+k\cwt\eoz;\\
 \textup{ii)}&&  e(O(2k)) &= k\cwd(\tau(\iota^{-2})\cxwt + e^{-2}\kappa\cxwd)
 =k\eoz;\\
 \textup{iii)}&& e(\chi O(2k+1)) &= \cxwd(1 + k\tau(1) + ke^{-2}\kappa\cxwt\cwd)
 =\cxwd+k\cxwt\eoz;\\
 \textup{iv)}&& e(\chi O(2k)) &= k\tau(1)\cwt\cxwd + e^2
 =\exoz+(k-1)\tau(\iota^2c).
\end{alignat*}
\end{proposition}

\begin{proof}
In each sequence of equalities, the first was proved in \cite[Proposition 6.5]{CHTFiniteProjSpace}, while the second equality can be obtained by combining the definitions of $Q$ and $\chi Q$ with the Frobenius relation $b\tau(a)=\tau(\rho(b)a)$ (equation (\ref{eqn:Frobenius})) and the values
of $\rho$ given in (\ref{eqn:restrictions}) in the appendix.

For example, in (i), we have
\[
    \tau(1)\cwd =\tau(\zeta c) =  \tau(c)\cwt,
\]
which, with the definition of $Q$, gives the claimed equality.
The calculations in (ii) and (iii) are similar.

Finally, for (iv), we first need to use one of the relations in the equivariant cohomology of projective space, namely
\begin{align}\label{eqn tensor}
\cwt\cxwd-(1-\kappa)\cxwt\cwd=e^2,
\end{align}
to replace $e^2$. The equality then follows by recalling that $\kappa=2-\tau(1)$  and making use of the equalities $\rho(\xi)=\iota^2$, $\tau(1)\cwt\cxwd=\tau(1)\cxwt\cwd$, and $\xi\cdot e^{-2}\kappa=0$.
\end{proof}

Let us now make more explicit the subdivision of the line bundles over $\Xpq pq$. For every such line bundle  $L$ we say that it is of type $\dagger\in \{\I,\II,\III,\IV\}$ if it is isomorphic to one of the line bundles appearing in the corresponding roman numeral of the previous Proposition. It is also convenient to apply this subdivision to the line bundle summands of $F = L_1\dirsum\cdots\dirsum L_n$. 
Let $\mathcal{F} = \{L_1,\ldots,L_n\}$ and let $\mathcal{F}_\dagger$ be the subset containing the summands of type $\dagger$. We set 
\begin{align*}
n_\dagger&:=|\mathcal{F}_\dagger|\,,\\ 
F_\dagger&:=\bigoplus_{L_i\in \mathcal{F}_\dagger} L_i\text{\quad and }\\
d_\dagger&:=\prod_{L_i\in \mathcal{F}_\dagger}d_i.
\end{align*}
By convention, we set $d_\dagger=1$ whenever $\mathcal{F}_\dagger=\emptyset$.
As direct consequences of these notations one has 
\begin{align*}
n&=n_\I+n_\II+n_\III+n_\IV \text{\quad and\ }\\
d&=d_\I d_\II d_\III d_\IV \,.
\end{align*}

As in \cite{CHTAlgebraic}, our calculations will depend only on the triples of ranks $(n,n_0,n_1)$ and
degrees $(\Delta,\Delta_0,\Delta_1)$, where
\begin{align*}
    n_0 &= n_\I + n_\II \\
    n_1 &= n_\II + n_\III \\
    \Delta &= d \\
    \Delta_0 &= 
        \begin{cases}
            d_\I d_\II & \text{if\ \,$n_0 < p$} \\
            0 & \text{if\,\  $n_0 \geq p$}
        \end{cases} \\
    \Delta_1 &= 
        \begin{cases}
            d_\II d_\III & \text{if\ \,$n_1 < q$} \\
            0 & \text{if\,\  $n_1 \geq q$.}
        \end{cases}
\end{align*}
Here, $n_0$ is the rank of the fixed-point bundle $F^\GG$ restricted to $\Xp p$ while
$\Delta_0$ is its nonequivariant degree when the rank is less than $p$. Similarly,
$n_1$ is the rank of $F^\GG$ restricted to $\Xq q$ and $\Delta_1$ is its nonequivariant degree.
Because $d_\I$ and $d_\III$ are always odd, $\Delta_0$ and $\Delta_1$ have the same parity as $d_\II$;
in particular, they have the same parity, which is even if and only if $n_\II > 0$.

In order for the Euler class to have a meaningful geometric interpretation, we impose the following
conditions, as we did in \cite{CHTAlgebraic}.

\begin{context}\label{context}
With notation as above, we assume that\begin{align*}
    n &< p + q \\
    n - q &\leq n_0 \leq n \\
    \mathllap{\text{and}\qquad} n - p &\leq n_1 \leq n.
\end{align*}
\end{context}


We will give two versions of B\'ezout's theorem, one in terms of the values above
and one in terms of the following values.
Thinking about the intersection of the hypersurfaces determined by our line bundles,
these are its nonequivariant affine dimension and the affine dimensions of its fixed sets:
\begin{align*}
    m &= p+q - n \\
    m_0 &= p - n_0 \\
    m_1 &= q - n_1.
\end{align*}
Note that these depend only on the triple $(n,n_0,n_1)$ and the values $p$ and $q$.
The assumptions in context~\ref{context} are equivalent to the inequalities
\begin{align*}
    0 &< m \\
    0&\leq m-m_0 \leq q \\
    \mathllap{\text{and}\qquad} 0 &\leq m-m_1 \leq p.
\end{align*}
We will also need the quantity
\[
\ell=n-n_0-n_1=m_0+m_1-m.
\]
Our version of B\'ezout's theorem breaks into cases depending on whether $\ell$ is positive or negative.

In order to obtain $e(F)$ it is convenient to first identify its factors $e(F_\dagger)$.
We start our work with a few computational lemmas. 
The first lemmas will enable us to rewrite expressions involving powers of $\eoz$.

\begin{lemma}  \label{lem Q powers}
For $k\geq 0$ one has 
\[
    Q^k = 2^{k-1}\Big(\tau(\cd[k])+e^{-2k}\kappa\, \cwd^{\;k}\,\cxwd^{\;k}\Big).
\]
If $k\geq 1$ we may also write this as
\[
    Q^k = 2^{k-1}\tau(\cd[k]) + (e^{-2}\kappa)^k \cwd^{\;k}\,\cxwd^{\;k}.
\]
\end{lemma}

\begin{proof}
The proof of the first equality is by induction on $k$, with the base case being $k=0$, for which the statement reduces to 
\[
    Q^0=1=\frac{1}{2}\cdot 2=\frac{1}{2}\Big(\tau(1)-\tau(1)+2\Big)=\frac{1}{2}\Big(\tau(1)+\kappa\Big).
\]

For the inductive step, we have
\begin{align*}
    Q^{k+1} &= 2^{k-1}\Big(\tau(\cd[k])+e^{-2k}\kappa \cwd^{\;k}\cxwd^{\;k}\Big) \cdot \Big(\tau(\cd)+e^{-2}\kappa\, \cwd\,\cxwd\Big)\\
    &= 2^{k-1}\Big(\tau(\cd[k])\tau(\cd)+2e^{-2k-2}\kappa\,\cwd^{\;k+1}\,\cxwd^{\;k+1}\Big) \\
    &= 2^k\Big(\tau(\cd[k+1]) + e^{-2(k+1)}\kappa\,\cwd^{\;k+1}\,\cxwd^{\;k+1}\Big).
\end{align*}
Here, we are using the fact that $\rho(e^{-2j}\kappa) = 0$, so the Frobenius relation (\ref{eqn:Frobenius}) implies that
$\tau(a)\cdot e^{-2j}\kappa = 0$ for any $a$. Further, $\rho\tau(a) = 2a$ for any $a$, so the Frobenius relation
implies that $\tau(\cd[k])\tau(\cd) = 2\tau(\cd[k+1])$. 
Finally, we have the relation $e^{-j}\kappa\cdot e^{-k}\kappa = 2e^{-j-k}\kappa$ in the cohomology of a point.

When $k\geq 1$, we have $(e^{-2}\kappa)^k = 2^{k-1}e^{-2k}\kappa$ from the relation just mentioned, hence
the second equality in the statement of the lemma.
\end{proof}

\begin{lemma}\label{lemma Q conversion}
If $i > k \geq 0$, then
\begin{alignat*}{3}
    \textup{ i)}\quad&& \cxwt^{i}\eoz^{k} &= \cxwt^{i-k}\cdot 2^{k}\cxwd^{\;k} & \qquad \text{and} \\
    \textup{ ii)}\quad&& \cwt^i\eoz^k &=\cwt^{i-k} \cdot 2^k \cwd^{\;k}.
\intertext{If $0 < i \leq k$, then}
    \textup{ iii)}\quad&& \cxwt^i\eoz^k &= \cxwt\cdot 2^{i-1}\,\cxwd^{\;i-1}\eoz^{k-i+1} & \qquad\text{and} \\
    \textup{ iv)}\quad&& \cwt^i\eoz^k &= \cwt \cdot 2^{i-1}\,\cwd^{\;i-1}\eoz^{k-i+1}.
\end{alignat*}
\end{lemma}

\begin{proof}
These equations follow by induction from the equalities
\begin{align}
    \cxwt^2 Q &= 2\cxwt\cxwd  \rlap{\text{\qquad and}} \label{eqn:zeta02Q} \\
    \cwt^2 Q &= 2\cwt\cwd. \label{eqn:zeta12Q}
\end{align}
To see (\ref{eqn:zeta02Q}), we have
\begin{align*}
    \cxwt^2 Q
    &= \cxwt^2 (\tau(c) + e^{-2}\kappa\cwd\cxwd) \\
    &= \tau(\iota^4\zeta^{-2} c) + e^{-2}\kappa\,\cxwt^2\,\cwd\cxwd \\
    &= \tau(\iota^4\zeta^{-2} c) + e^{-2}\kappa\cxwt\cxwd((1-\kappa)\cwt\cxwd + e^2) \\
    &= \tau(\iota^4\zeta^{-2} c) + (1-\kappa)e^{-2}\kappa \,\xi \,\cxwd^{\;2} + \kappa\,\cxwt\,\cxwd \\
    &= \tau(\iota^4\zeta^{-2} c) + (2- \tau(1))\cxwt\cxwd \\
    &= \tau(\iota^4\zeta^{-2} c) + 2\cxwt\cxwd - \tau(\iota^4\zeta^{-2} c) \\
    &= 2\cxwt\,\cxwd.
\end{align*}
The proof of (\ref{eqn:zeta12Q}) is similar.
\end{proof}

\begin{lemma}\label{lem:xiQ}
$\xi Q^k = 2^{k-1}\tau(\iota^2 c^k)$.
\end{lemma}

\begin{proof}
By Lemma~\ref{lem Q powers}, the fact that $\xi\cdot e^{-k}\kappa = 0$, and the Frobenius relation, we have
\[
    \xi Q^k = \xi \cdot 2^{k-1}\Big(\tau(c^k)+e^{-2k}\kappa\, \cwd^{\:k}\,\cxwd^{\:k}\Big)=2^{k-1}\tau\big(\iota^2\,c^k\big).\qedhere
\]
\end{proof}

We also have some special cases in which we can simplify expressions further.

\begin{lemma}\label{lem:simplification}
In the cohomology of $\Xpq pq$,
the class $\cwd^{\;p-k}Q^k$ is infinitely divisible by $\cxwt$ and $\cxwd^{\;q-k}Q^k$ is infinitely divisible by $\cwt$,
and, for $\ell \geq 0$,
\begin{align*}
   Q^i\cdot \cwd^{\;p-k}Q^{k} &= 2^i\,\cxwd^{\;i} \cdot \cxwt^{-i}\,\cwd^{\;p-k} Q^k &&\text{if $0 \leq k \leq p$, and} \\
   Q^i \cdot\cxwd^{\;q-k}Q^{k} &= 2^i\, \cwd^{\;i}\cdot \cwt^{-i}\,\cxwd^{\;q-k} Q^k &&\text{if $0\leq k \leq q$.}
\end{align*}
\end{lemma}

\begin{proof}
We start with the fact that $\cwd^{\;p-k}Q^k$ is divisible by $\cxwt$, which we know is
true if $k=0$. If $k>0$, then, by Lemma~\ref{lem Q powers}, we have
\begin{align*}
    \cwd^{\;p-k}Q^k &= 2^{k-1}\cwd^{\;p-k}(\tau(c^k) + e^{-2k}\kappa\cwd^{\;k}\cxwd^{\;k}) \\
    &= 2^{k-1}\cwd^{\;p}(\tau(\zeta^{-k}) + e^{-2k}\kappa\cxwd^{\;k})
\end{align*}
which is a multiple of $\cwd^{\;p}$, hence divisible by $\cxwt$.
A similar calculation shows that $\cxwd^{\;q-k}Q^k$ is divisible by $\cwt$.

For the equalities, we first have
\begin{align*}
    \cxwt \cwd^{\;p-k}Q^{k+1}
    &= 2^k \cxwt \cwd^{\;p-k} (\tau(c^{k+1}) + e^{-2(k+1)}\kappa \cwd^{\;k+1}\cxwd^{\;k+1}) \\
    &= 2^k(\cwd^{\;p-k}\cxwd\tau(c^k) + e^{-2(k+1)}\kappa\cwd^{\;p}\cxwd^{\;k+1}\cdot \cxwt\cwd) \\
    &= 2^k\Bigl(\cwd^{\;p-k}\cxwd\tau(c^k) \\
    &\qquad\qquad + e^{-2(k+1)}\kappa\cwd^{\;p}\cxwd^{\;k+1}((1-\kappa)\cwt\cxwd + e^2) \Bigr) \\
    &= 2^k\Bigl(\cwd^{\;p-k}\cxwd\tau(c^k) \\
    &\qquad\qquad - e^{-2(k+1)}\kappa\xi\cxwt^{-1}\cwd^{\;p}\cxwd^{\;k+2} + e^{-2k}\kappa \cwd^{\;p}\cxwd^{\;k+1} \Bigr) \\
    &= 2^k(\cwd^{\;p-k}\cxwd\tau(c^k) + e^{-2k}\kappa \cwd^{\;p}\cxwd^{\;k+1}) \\
    &= 2\cwd^{\;p-k}\cxwd \cdot 2^{k-1}(\tau(c^k) + e^{-2k}\kappa\cwd^{\;k}\cxwd^{\;k}) \\
    &= 2\cwd^{\;p-k}\cxwd Q^k.
\end{align*}
From this it follows that $Q\cdot \cwd^{\;p-k}Q^{k} = 2\cxwd \cdot \cxwt^{-1}\,\cwd^{\;p-k} Q^k$
and the statement for $i > 1$ follows by iteration.
The second equality follows similarly.
\end{proof}

The next lemma will allow us to deal with powers of $\exoz$, for which we need some notation: 
for $k\in\ZZ$ an integer, let us denote by $\rem{k}\in\{0,1\}$ the remainder of the division of $k$ by 2 and by $\beta(k)$ the number of $1$s appearing in the binary representation of $k$. 

\begin{lemma} \label{lemma XQ powers}
For $k \in \NN$ one has
$$\exoz^k=\sum_{j=0}^{k}\rem{\binom{k}{j}}\:(\cxwt \cwd)^j(\cwt \cxwd)^{k-j}+\frac{2^{k}-2^{\beta(k)}}{2}\tau(\iota^{2k}c^k).$$
\end{lemma}

\begin{proof}
Since the statement is trivial for $k=0$, we can assume $k\geq 1$. 
\begin{align*}
\exoz^k &= (\cxwt\cwd + \cwt\cxwd)^k \\
&=\sum_{j=0}^{k}\binom{k}{j}(\cxwt \cwd)^j(\cwt \cxwd)^{k-j} \\
&=\sum_{j=0}^{k}\rem{\binom{k}{j}}(\cxwt \cwd)^j(\cwt \cxwd)^{k-j}+
\sum_{j=0}^{k}\left[\binom{k}{j}-\rem{\binom{k}{j}}\right](\cxwt \cwd)^j(\cwt \cxwd)^{k-j}\\
&=\sum_{j=0}^{k}\rem{\binom{k}{j}}(\cxwt \cwd)^j(\cwt \cxwd)^{k-j}+
\sum_{j=1}^{k-1}\left[\binom{k}{j}-\rem{\binom{k}{j}}\right](\cxwt \cwd)^j(\cwt \cxwd)^{k-j}\\
&=\sum_{j=0}^{k}\rem{\binom{k}{j}}(\cxwt \cwd)^j(\cwt \cxwd)^{k-j}+
2\xi\sum_{j=1}^{k-1}\frac{\binom{k}{j}-\rem{\binom{k}{j}}}2\cxwt^{j-1} \cwd^j\cwt^{k-j-1} \cxwd^{k-j}\\
&\stackrel{(*)}=\sum_{j=0}^{k}\rem{\binom{k}{j}}(\cxwt \cwd)^{\:j}(\cwt \cxwd)^{\:k-j}+
\tau(\iota^2)\sum_{j=1}^{k-1}\frac{\binom{k}{j}-\rem{\binom{k}{j}}}2\cxwt^{j-1} \cwd^{\:j}\cwt^{k-j-1} \cxwd^{\:k-j}\\
&=\sum_{j=0}^{k}\rem{\binom{k}{j}}(\cxwt \cwd)^j(\cwt \cxwd)^{k-j}
+
\sum_{j=1}^{k-1}\frac{\binom{k}{j}-\rem{\binom{k}{j}}}2\tau(\iota^{2k}c^k)\\
&=\sum_{j=0}^{k}\rem{\binom{k}{j}}(\cxwt \cwd)^j(\cwt \cxwd)^{k-j}
+
\frac{2^k-2^{\beta(k)}}2\tau(\iota^{2k}c^k)
\end{align*}
In $(*)$ we used that $2\xi=\tau(\iota^2)$, while the last equality follows from \cite[Theorem 2]{BinomialFine}. The other steps are just algebraic manipulations and, in the penultimate step, an application of the Frobenius relation.
\end{proof}

The final lemma we need simplifies powers of $Q\cdot \chi Q$.

\begin{lemma}\label{lem:Q xQ}
If $k\geq 1$, then
$(Q\cdot\chi Q)^k = 2^{k-1}(2^k-1)\tau(\iota^{2k} c^{2k}) + 2^k\cwd^{\;k}\cxwd^{\;k}$.
\end{lemma}

\begin{proof}
We use the relation $\chi Q = \cxwt\cwd + \cwt\cxwd = \tau(\iota^2 c) + e^2$
and recall that $\kappa = 2 - \tau(1)$:
\begin{align*}
    Q\cdot \chi Q
    &= \big(\tau(c) + e^{-2}\kappa\cwd\cxwd\big)\big(\cxwt\cwd + \cwt\cxwd\big) \\
    &= \big(\tau(c) + e^{-2}\kappa\cwd\cxwd\big)\big(\tau(\iota^2 c) + e^2\big) \\
    &= 2\tau(\iota^2 c^2) + \kappa \cwd\cxwd \\
    &= 2\tau(\iota^2 c^2) + 2\cwd\cxwd - \tau(\iota^2 c^2) \\
    &= \tau(\iota^2 c^2) + 2\cwd\cxwd.
\end{align*}
Here we used the Frobenius relation to conclude that $\tau(c)e^2$ and  $e^{-2}\kappa \cdot \tau(\iota^2c)$ both vanish. This gives
\begin{align*}
    (Q\cdot \chi Q)^k 
    &= \bigl(\tau(\iota^2 c^2) + 2\cwd\cxwd\bigr)^k \\
    &= \sum_{j=0}^{k-1}\binom{k}{j}\tau(\iota^2 c^2)^{k-j}\cdot 2^j\cwd^{\;j}\cxwd^{\;j} + 2^k\cwd^{\;k}\cxwd^{\;k} \\
    &= \sum_{j=0}^{k-1}\binom{k}{j}2^{k-j-1}2^j\tau(\iota^{2(k-j)}c^{2(k-j)})\cwd^{\;j}\cxwd^{\;j} + 2^k\cwd^{\;k}\cxwd^{\;k} \\
    &= 2^{k-1}\sum_{j=0}^{k-1}\binom{k}{j}\tau(\iota^{2k}c^{2k}) + 2^k\cwd^{\;k}\cxwd^{\;k} \\
    &= 2^{k-1}(2^k-1)\tau(\iota^{2k} c^{2k}) + 2^k\cwd^{\;k}\cxwd^{\;k}. \qedhere
\end{align*}
\end{proof}

We now start to compute Euler classes.
The following proposition is essentially a restatement of the lemmas leading to the proof of \cite[Proposition 6.12]{CHTFiniteProjSpace}.
We use the notation discussed before Context~\ref{context}.

\begin{proposition}\label{prop I-IV}
\begin{align*}
\textup{ i)}&&  e(F_\I)&
=\cwd^{\:n_\I}+\frac{d_\I-1}{2}\cwd^{\:n_\I-1}\cwt \eoz;\\
\textup{ ii)}&& e(F_\II)&=\frac{d_\II}{2^{n_\II}}\eoz^{n_\II} ;\\
\textup{ iii)}&&  e(F_\III)&
=\cxwd^{\:n_\III}+\frac{d_\III-1}{2}\,\cxwd^{\:n_\III-1}\cxwt\eoz;\\
\textup{ iv)}&&  e(F_\IV)&=\exoz^{n_\IV}+\frac{d_{\IV}-2^{n_\IV}}{2}\tau\big(\iota^{2n_\IV}c^{n_\IV}\big)\\
&& &=\sum_{j=0}^{n_\IV}\rem{\binom{n_\IV}{j}}\:(\cxwt \cwd)^j(\cwt \cxwd)^{n_\IV-j}
+
\frac{d_{\IV}-2^{\beta(n_\IV)}}{2}\tau\big(\iota^{2n_\IV}c^{n_\IV}\big).
\end{align*}
\end{proposition}
\begin{proof}
(i) We proceed by induction, with the case $n_\I = 0$ being true because $d_\I = 1$ in that case.

Inductive step. We assume the statement to be true for $n_\I-1$ and set $\widehat{F}_\I=L_1\oplus\cdots\oplus L_{n_\I-1}$ and $\widehat{d}_\I:=\prod_{i=1}^{n_\I-1}d_i$. Using part (i) of Proposition~\ref{prop: base case}, we have
\begin{align*}
e(F_\I)&=e(\widehat{F}_\I)\cdot e(L_{n_\I})\\
&=\Big(\cwd^{\:n_\I-1}+\frac{\widehat{d}_\I -1}{2}\cwd^{\:n_\I-2}\cwt\eoz\Big)\cdot\Big(\cwd+\frac{d_{n_\I}-1}2\cwt\eoz\Big)\\
&=\cwd^{\:n_\I}+\frac{\widehat{d}_\I-1}{2}\cwd^{\:n_\I-1}\cwt\eoz+ \frac{d_{n_\I}-1}2\cwd^{\:n_\I-1}\cwt\eoz+\frac{d_\I-d_{n_\I}-\widehat{d}_\I+1}{4}\cwd^{\:n_\I-2}\cwt^{2}\eoz^{2}\\
&=\cwd^{\:n_\I}+\frac{\widehat{d}_\I-1}{2}\cwd^{\:n_\I-1}\cwt\eoz+ \frac{d_{n_\I}-1}2\cwd^{\:n_\I-1}\cwt\eoz+\frac{d_\I-d_{n_\I}-\widehat{d}_\I+1}{2}\cwd^{\:n_\I-1}\cwt\eoz\\
&=\cwd^{\:n_\I}+\left(\frac{\widehat{d}_\I-1}{2}+ \frac{d_{n_\I}-1}{2}+\frac{d_\I-d_{n_\I}-\widehat{d}_\I+1}{2}\right)\cwd^{\:n_\I-1}\cwt\eoz\\
&=\cwd^{\:n_\I}+\left(\frac{d_\I-1}{2}\right)
\cwd^{\:n_\I-1}\cwt\eoz
\end{align*}
With the exception of the usage of the inductive hypothesis in the second step and of part (iv) of Lemma \ref{lemma Q conversion} in the fourth, the remaining manipulations are elementary.

(ii) This is a straightforward consequence of part (ii) of Proposition \ref{prop: base case}.

(iii) The proof is similar to that of (i), or can be reduced to (i) by applying $\chi$.

(iv) The second equality is a straightforward consequence of Lemma \ref{lemma XQ powers}. We prove the first equality by induction on $n_{\IV}$. 
The base case is $n_\IV = 0$, which is true because $d_\IV = 1$ in that case.
For the inductive step, We use \textit{mutatis mutandis} the notations introduced in part (i).
\begin{align*}
e(F_\IV)&=e(\widehat{F}_\IV)\cdot e(L_{n_\IV}) \\
&=\left[\exoz^{n_\IV-1}+\frac{\widehat{d}_\IV-2^{n_{\IV}-1}}2\tau\Big(\iota^{2n_\IV-2}\,c^{n_\IV-1}\Big)\right]
\left[\exoz+\frac{d_{n_\IV}-2}2\tau\big(\iota^2\,c\big)\right]\\
&=\exoz^{n_\IV}+
\exoz^{n_\IV-1}\cdot \frac{d_{n_\IV}-2}2\cdot \tau\big(\iota^2c\big)\\
&\qquad {}+
\exoz\cdot \frac{\widehat{d}_\IV-2^{n_{\IV}-1}}2\cdot \tau\Big(\iota^{2n_\IV-2}\,c^{n_\IV-1}\Big)\\
&\qquad {}+
\frac{d_\IV-2\widehat{d}_\IV-2^{n_{\IV}-1}d_{n_\IV}+2^{n_\IV}}4\cdot \tau\Big(\iota^{2n_\IV-2}\,c^{n_\IV-1}\Big)\cdot \tau\big(\iota^2\,c\big)\\
&=\exoz^{n_\IV}+
\Big(\cxwt \cwd+\cwt \cxwd\Big)^{n_\IV-1}\cdot \frac{d_{n_\IV}-2}2\cdot \tau\big(\iota^2c\big) \\
&\qquad{} +
\Big(\cxwt \cwd+\cwt \cxwd\Big)\cdot \frac{\widehat{d}_\IV-2^{n_{\IV}-1}}2\cdot \tau\Big(\iota^{2n_\IV-2}\,c^{n_\IV-1}\Big)\\
&\qquad {}+
\frac{d_\IV-2\widehat{ d}_\IV-2^{n_{\IV}-1}d_{n_\IV}+2^{n_\IV}}4\cdot 2\tau\Big(\iota^{2n_\IV}\,c^{n_\IV}\Big)
\intertext{We now collect all of the terms involving $\tau$, using the Frobenius relation and the fact that
$\rho(\cxwt \cwd+\cwt \cxwd)^k = (2\iota^2 c)^k$:}
e(F_\IV) 
&=\exoz^{n_\IV}+
\frac{d_{n_\IV}-2}2 \cdot 2^{n_{\IV}-1}\cdot\tau\Big(\iota^{2n_\IV}\,c^{n_\IV}\Big) \\
&\qquad{} +
\frac{\widehat{d}_\IV-2^{n_{\IV}-1}}2\cdot 2\tau\Big(\iota^{2n_\IV}\,c^{n_\IV}\Big)\\
&\qquad {}+
\frac{d_\IV-2\widehat{ d}_\IV-2^{n_{\IV}-1}d_{n_\IV}+2^{n_\IV}}{2}\cdot \tau\Big(\iota^{2n_\IV}\,c^{n_\IV}\Big) \\
&=\exoz^{n_\IV}+
    \frac{d_\IV-2^{n_{\IV}}}2 \cdot \tau\Big(\iota^{2n_\IV}\,c^{n_\IV}\Big)\qedhere
\end{align*}
\end{proof}

We now want to multiply these expressions together to find the Euler class of $F$.
The results break down into cases depending on the relative numbers of line bundles of types $\II$ and $\IV$
we have, so we begin with the following computation.

\begin{lemma}\label{lem:IandIII}
\begin{align*}
e(F_\I\oplus F_\III)&=
\cwd^{\:n_\I}\cxwd^{\:n_\III}
+
\frac{d_\I-1}{2}\cwd^{\:n_\I-1}\,\cxwd^{\:n_\III}\cwt \eoz
+
\frac{d_\III-1}{2}\,\cwd^{\:n_\I}\,\cxwd^{\:n_\III-1}\cxwt\eoz \\
&\ \ +
\frac{(d_\I-1)(d_\III-1)}{2}\tau\Big(\iota^{2n_\III}\,\zeta^{n_\I-n_\III}\,c^{n_\I+n_\III}\Big)\notag.
\end{align*}
\end{lemma}

\begin{proof}
By Proposition \ref{prop I-IV} and Lemma~\ref{lem:xiQ}, we have 
\begin{align*}
e(F_\I\oplus F_\III)&=e(F_\I)\cdot e(F_\III) \\
&=\left[
\cwd^{\:n_\I}+\frac{d_\I-1}{2}\cwd^{\:n_\I-1}\cwt \eoz
\right]
\left[
\cxwd^{\:n_\III}+\frac{d_\III-1}{2}\,\cxwd^{\:n_\III-1}\cxwt\eoz
\right] \\
&=\cwd^{\:n_\I}\cxwd^{\:n_\III}
+
\frac{d_\I-1}{2}\cwd^{\:n_\I-1}\,\cxwd^{\:n_\III}\cwt \eoz
+
\frac{d_\III-1}{2}\,\cwd^{\:n_\I}\,\cxwd^{\:n_\III-1}\cxwt\eoz \\
&\qquad +
\frac{(d_\I-1)(d_\III-1)}{4} \cwd^{\:n_\I-1}\,\cxwd^{\:n_\III-1}\xi \eoz^2 \\
&= 
\cwd^{\:n_\I}\cxwd^{\:n_\III}
+
\frac{d_\I-1}{2}\cwd^{\:n_\I-1}\,\cxwd^{\:n_\III}\cwt \eoz
+
\frac{d_\III-1}{2}\,\cwd^{\:n_\I}\,\cxwd^{\:n_\III-1}\cxwt\eoz \\
&\qquad +
\frac{(d_\I-1)(d_\III-1)}{2}\tau\Big(\iota^{2n_\III}\,\zeta^{n_\I-n_\III}\,c^{n_\I+n_\III}\Big).
\qedhere
\end{align*}
\end{proof}

We now compute $e(F)$, breaking the result into two cases, corresponding to the following two propositions. 
We express the results using the notations introduced around Context~\ref{context}.
In the next two proposition we will make use of the following quantities to simplify exponents, let
\begin{align*}
    k_0 &= q-(m-m_0) \\
    k_1 &= p-(m-m_1)
\end{align*}
and note that these are both nonnegative by assumption.
Note that these results depend only on the
ranks $(n,n_0,n_1)$ and degrees $(\Delta,\Delta_0,\Delta_1)$.


\begin{proposition}\label{prop:nIV less nII}
Suppose that $\ell = m_0 + m_1 - m \leq 0$.
Let $\Delta_{\min} = \min(\Delta_0,\Delta_1)$ and let $\Delta_{\max} = \max(\Delta_0,\Delta_1)$.
Then
\begin{align*}
    e(F) &= \frac{\Delta_{\min}}{2^{|\ell|}}\cwd^{\;k_1}\cxwd^{\;k_0} Q^{|\ell|} 
    + \frac{\Delta_0-\Delta_{\min}}{2^{|\ell| + 1}} \cwd^{\;k_1-1}\cxwd^{\;k_0} \cwt Q^{|\ell|+1} \\
    &\qquad {}+ \frac{\Delta_1-\Delta_{\min}}{2^{|\ell| + 1}} \cwd^{\;k_1}\cxwd^{\;k_0-1} \cxwt Q^{|\ell|+1} 
    + \frac{\Delta - \Delta_{\max}}{2} \tau(\iota^{2k_0}\zeta^{k_1-k_0} \cd[p+q-m]).
\end{align*}
If $m_0 \leq 0$, this simplifies to
\[ 
    e(F) = \frac{\Delta_1}{2^{m-m_1}}\cxwt^{m_0}\cwd^{\;k_1}\cxwd^{\;q-m}Q^{m-m_1} 
            + \frac{\Delta-\Delta_1}{2}\tau(\iota^{2k_0}\zeta^{k_1-k_0} \cd[p+q-m]).
\]
If $m_1 \leq 0$, it simplifies to
\[
    e(F) = \frac{\Delta_0}{2^{m-m_0}}\cwt^{m_1}\cwd^{\;p-m}\cxwd^{\;k_0}Q^{m-m_0} 
            + \frac{\Delta-\Delta_0}{2}\tau(\iota^{2k_0}\zeta^{k_1-k_0} \cd[p+q-m]).
\]
If both $m_0\leq 0$ and $m_1\leq 0$, it simplifies to
\[
    e(F) = \frac{\Delta}{2}\tau(\iota^{2k_0}\zeta^{k_1-k_0} \cd[p+q-m]).
\]
\end{proposition}

\begin{proof}
We have $\ell = n_\IV - n_\II$, so our assumption is equivalent to $n_\IV \leq n_\II$.
We start with the following computation, where we use that $\rho(Q) = 2\cd$.
\begin{align*}
    e(F_\II\dirsum F_\IV)
    &= e(F_\II)e(F_\IV) \\
    &= \frac{d_\II}{2^{n_\II}} Q^{n_\II} \cdot \left(\chi Q^{n_\IV} + \frac{d_\IV - 2^{n_\IV}}{2}\tau(\iota^{2n_\IV} \cd[n_\IV])\right) \\
    &= \frac{d_\II}{2^{n_\II}}(Q\cdot \chi Q)^{n_\IV} Q^{n_\II-n_\IV} + \frac{d_\II(d_\IV-2^{n_\IV})}{2}\tau(\iota^{2n_\IV}\cd[n_\II+n_\IV]) \\
    &= \frac{d_\II}{2^{n_\II}}Q^{n_\II-n_\IV}\bigl(2^{n_\IV-1}(2^{n_\IV}-1)\tau(\iota^{2n_\IV} \cd[2n_\IV]) + 2^{n_\IV}\cwd^{\;n_\IV}\cxwd^{\;n_\IV}\bigr) \\
    &\qquad + \frac{d_\II(d_\IV-2^{n_\IV})}{2}\tau(\iota^{2n_\IV}\cd[n_\II+n_\IV]) \\
    &= \frac{d_\II}{2^{n_\II}}\bigl(2^{n_\II-1}(2^{n_\IV}-1)\tau(\iota^{2n_\IV}\cd[n_\II+n_\IV]) + 2^{n_\IV}\cwd^{\;n_\IV}\cxwd^{\;n_\IV}Q^{n_\II-n_\IV}\bigr) \\
    &\qquad + \frac{d_\II(d_\IV-2^{n_\IV})}{2}\tau(\iota^{2n_\IV}\cd[n_\II+n_\IV]) \\
    &= \frac{d_\II}{2^{n_\II-n_\IV}}\cwd^{\;n_\IV}\cxwd^{\;n_\IV}Q^{n_\II-n_\IV}
        + \frac{d_\II(d_\IV-1)}{2}\tau(\iota^{2n_\IV}\cd[n_\II+n_\IV])
\end{align*}
Now we multiply by $e(F_\I\dirsum F_\III)$ using the result of Lemma~\ref{lem:IandIII}.
We suppress the calculation of all of the terms involving $\tau(\iota^{2(n_\III+n_\IV)}\zeta^{n_\I-n_\III}\cd[n])$.
\begin{align*}
    e(F) &= e(F_\I\dirsum F_\III)e(F_\II\dirsum F_\IV) \\
    &= \frac{d_\II}{2^{n_\II-n_\IV}}\cwd^{\;n_\I+n_\IV}\cxwd^{\;n_\III+n_\IV}Q^{n_\II-n_\IV} \\
    &\qquad + \frac{d_\II(d_\I-1)}{2^{n_\II-n_\IV+1}}\cwd^{\;n_\I+n_\IV-1}\cxwd^{\;n_\III+n_\IV}\cwt Q^{n_\II-n_\IV+1} \\
    &\qquad + \frac{d_\II(d_\III-1)}{2^{n_\II-n_\IV+1}}\cwd^{\;n_\I+n_\IV}\cxwd^{\;n_\III+n_\IV-1}\cxwt Q^{n_\II-n_\IV+1} \\
    &\qquad + \frac{d-d_\II(d_\I+d_\III-1)}{2}\tau(\iota^{2(n_\III+n_\IV)}\zeta^{n_\I-n_\III}\cd[n])
\end{align*}
Now recall that $\Delta_0 = d_\I d_\II$ and $\Delta_1 = d_\II d_\III$. We can write the sum of the middle two terms above
as
\begin{align*}
    \frac{d_\II(d_\I-1)}{2^{n_\II-n_\IV+1}}&\cwd^{\;n_\I+n_\IV-1}\cxwd^{\;n_\III+n_\IV}\cwt Q^{n_\II-n_\IV+1} \\
    & + \frac{d_\II(d_\III-1)}{2^{n_\II-n_\IV+1}}\cwd^{\;n_\I+n_\IV}\cxwd^{\;n_\III+n_\IV-1}\cxwt Q^{n_\II-n_\IV+1} \\
    &\qquad = \frac{\Delta_0 - d_\II}{2^{n_\II-n_\IV+1}}\cwd^{\;n_\I+n_\IV-1}\cxwd^{\;n_\III+n_\IV}\cwt Q^{n_\II-n_\IV+1} \\
    &\qquad\qquad + \frac{\Delta_1 - d_\II}{2^{n_\II-n_\IV+1}}\cwd^{\;n_\I+n_\IV}\cxwd^{\;n_\III+n_\IV-1}\cxwt Q^{n_\II-n_\IV+1} \\
    &\qquad = \frac{\Delta_0 - \Delta_{\min}}{2^{n_\II-n_\IV+1}}\cwd^{\;n_\I+n_\IV-1}\cxwd^{\;n_\III+n_\IV}\cwt Q^{n_\II-n_\IV+1} \\
    &\qquad\qquad + \frac{\Delta_1 - \Delta_{\min}}{2^{n_\II-n_\IV+1}}\cwd^{\;n_\I+n_\IV}\cxwd^{\;n_\III+n_\IV-1}\cxwt Q^{n_\II-n_\IV+1} \\
    &\qquad\qquad + \frac{\Delta_{\min} - d_\II}{2^{n_\II-n_\IV+1}} \cwd^{\;n_\I+n_\IV-1}\cxwd^{\;n_\III+n_\IV-1} \chi Q \cdot Q^{n_\II-n_\IV+1} \\
    &\qquad = \frac{\Delta_0 - \Delta_{\min}}{2^{n_\II-n_\IV+1}}\cwd^{\;n_\I+n_\IV-1}\cxwd^{\;n_\III+n_\IV}\cwt Q^{n_\II-n_\IV+1} \\
    &\qquad\qquad + \frac{\Delta_1 - \Delta_{\min}}{2^{n_\II-n_\IV+1}}\cwd^{\;n_\I+n_\IV}\cxwd^{\;n_\III+n_\IV-1}\cxwt Q^{n_\II-n_\IV+1} \\
    &\qquad\qquad + \frac{\Delta_{\min} - d_\II}{2^{n_\II-n_\IV}} \cwd^{\;n_\I+n_\IV}\cxwd^{\;n_\III+n_\IV} Q^{n_\II-n_\IV} \\
    &\qquad\qquad + \frac{\Delta_{\min} - d_\II}{2}\tau(\iota^{2(n_\III+n_\IV)}\zeta^{n_\I-n_\III}\cd[n])
\end{align*}
Here, we used that $\chi Q=\cxwt\cwd+\cwt\cxwd$ in the second step, and Lemma \ref{lem:Q xQ} and the fact that $\rho(Q) = 2\cd$ in the last step.
Combining with the other terms gives us
\begin{align*}
    e(F)
    &= \frac{\Delta_{\min}}{2^{n_\II-n_\IV}}\cwd^{\;n_\I+n_\IV}\cxwd^{\;n_\III+n_\IV}Q^{n_\II-n_\IV} \\
    &\qquad + \frac{\Delta_0-\Delta_{\min}}{2^{n_\II-n_\IV+1}}\cwd^{\;n_\I+n_\IV-1}\cxwd^{\;n_\III+n_\IV}\cwt Q^{n_\II-n_\IV+1} \\
    &\qquad + \frac{\Delta_1-\Delta_{\min}}{2^{n_\II-n_\IV+1}}\cwd^{\;n_\I+n_\IV}\cxwd^{\;n_\III+n_\IV-1}\cxwt Q^{n_\II-n_\IV+1} \\
    &\qquad + \frac{d-d_\II(d_\I+d_\III)+\Delta_{\min}}{2}\tau(\iota^{2(n_\III+n_\IV)}\zeta^{n_\I-n_\III}\cd[n])
\end{align*}
We now notice that $d-d_\II(d_\I+d_\III)+\Delta_{\min} = d - \Delta_{\max} = \Delta - \Delta_{\max}$. On substituting
\begin{align*}
    n_\II - n_\IV &= n_0 + n_1 - n = |\ell| \\
    n_\I + n_\IV &= n - n_1 = k_1 \\
    n_\III + n_\IV &= n - n_0 = k_0 \\
    n_\I - n_\III &= n_0 - n_1 = k_1 - k_0,
\end{align*}
we get the first expression in the proposition.

If $m_0 \leq 0$, then $\Delta_0 = 0$, so $\Delta_{\min} = 0$ as well (this assumes that $\Delta_1\geq 0$,
the case $\Delta_1 < 0$ can be handled similarly), and the first two terms of the general formula
are both 0. 
We have $\ell + k_1 = p - m_0 \geq p$, so we can apply Lemma~\ref{lem:simplification} to rewrite the third term as
\begin{align*}
    \frac{\Delta_1-\Delta_{\min}}{2^{|\ell| + 1}} &\cwd^{\;k_1}\cxwd^{\;k_0-1} \cxwt Q^{|\ell|+1} \\
    &= \frac{\Delta_1}{2^{|\ell| + 1}} 2^{|\ell|+1+k_1-p}\cxwt^{-(|\ell|+1+k_1-p)+1}\cwd^{\;k_1}\cxwd^{\;k_0-1+|\ell|+1+k_1-p} Q^{p-k_1} \\
    &= \frac{\Delta_1}{2^{m-m_1}}\cxwt^{m_0}\cwd^{\;k_1}\cxwd^{\;q-m}Q^{m-m_1}.
\end{align*}
Similar simplifications occur if $m_1\leq 0$. If both $m_0\leq 0$ and $m_1\leq 0$, then $\Delta_0 = \Delta_1 = 0$
and all but the last term of the general formula are 0.
\end{proof}

\begin{proposition}\label{prop:nII less nIV} 
Suppose that $\ell> 0$.
Let $\epsilon = \overline \Delta_0$, so $\epsilon = 0$ if $\Delta_0$ and $\Delta_1$ are even
or $1$ if $\Delta_0$ and $\Delta_1$ are odd. Then
\begin{align*}
    e(F) &= \epsilon\xi\sum_{j=1}^{\ell-1}\rem{\binom{\ell}{j}}\cwd^{\;p-m_0+j}\;\cxwd^{\;k_0-j}\,\cxwt^{j-1}\,\cwt^{\;\ell-j-1} \\
    &\qquad {} + \Delta_0 \cwd^{\;p-m_0}\cxwd^{\;k_0}\cwt^{\;\ell}
        + \Delta_1 \cwd^{\;k_1}\cxwd^{\;q-m_1}\cxwt^{\;\ell}   \\
    &\qquad {}+ \frac{\Delta - \Delta_0 - \Delta_1 - \epsilon(2^{\beta(\ell)}-2)}{2}\tau(\iota^{2k_0}\zeta^{k_1-k_0}\cd[p+q-m]).
\end{align*}
\end{proposition}

\begin{proof}
The assumption that $\ell > 0$ is equivalent to $n_\II < n_\IV$.
We start with the case $\epsilon = 1$, which is equivalent to $n_\II = 0$.

We use the second representation of $e(F_\IV)$ from Proposition~\ref{prop I-IV} and multiply it by
the expression from Lemma~\ref{lem:IandIII}. Note that every term in the sum in $e(F_\IV)$, except for the first
and last terms, is divisible by $\xi$. This allows us to use Lemma~\ref{lem:xiQ} to simplify several of the terms
in the product. We use Lemma~\ref{lemma Q conversion} to simplify others. We get
\begin{align*}
    e(F) &= e(F_\I \dirsum F_\III)e(F_\IV) \\
    &= \cwd^{\:n_\I}\cxwd^{\:n_\III}\sum_{j=0}^{n_\IV}\rem{\binom{n_\IV}{j}}\:(\cxwt \cwd)^j(\cwt \cxwd)^{n_\IV-j} \\
    &\qquad + \frac{d_\I-1}{2}\cwd^{\:n_\I-1}\,\cxwd^{\:n_\III+n_\IV}\cwt^{n_\IV+1} \eoz \\
    &\qquad + \frac{d_\III-1}{2}\,\cwd^{\:n_\I+n_\IV}\,\cxwd^{\:n_\III-1}\cxwt^{n_\IV+1}\eoz \\
    &\qquad + \frac{d - d_\I - d_\III - 2^{\beta(n_\IV)}-2}{2}\tau\Big(\iota^{2(n_\III+n_\IV)}\,\zeta^{n_\I-n_\III}\,\cd[n_\I+n_\III+n_\IV]\Big) \\
    &= \xi \sum_{j=1}^{n_\IV-1}\rem{\binom{n_\IV}{j}} \cwd^{\;n_\I+j}\cxwd^{\;n_\III+n_\IV-j}\cxwt^{j-1}\cwt^{n_\IV-j-1} \\
    &\qquad + \cwd^{\;n_\I}\cxwd^{\;n_\III+n_\IV}\cwt^{n_\IV} + \cwd^{\;n_\I+n_\IV}\cxwd^{\;n_\III}\cxwt^{n_\IV} \\
    &\qquad + (d_\I-1)\cwd^{\;n_\I}\cxwd^{\;n_\III+n_\IV}\cwt^{n_\IV} + (d_\III-1)\cwd^{\;n_\I+n_\IV}\cxwd^{\;n_\III}\cxwt^{n_\IV} \\
    &\qquad + \frac{d - d_\I - d_\III - 2^{\beta(n_\IV)}-2}{2}\tau\Big(\iota^{2(n_\III+n_\IV)}\,\zeta^{n_\I-n_\III}\,\cd[n_\I+n_\III+n_\IV]\Big) \\
    &= \xi \sum_{j=1}^{n_\IV-1}\rem{\binom{n_\IV}{j}} \cwd^{\;n_\I+j}\cxwd^{\;n_\III+n_\IV-j}\cxwt^{j-1}\cwt^{n_\IV-j-1} \\
    &\qquad + d_\I \cwd^{\;n_\I}\cxwd^{\;n_\III+n_\IV}\cwt^{n_\IV} + d_\III \cwd^{\;n_\I+n_\IV}\cxwd^{\;n_\III}\cxwt^{n_\IV} \\
    &\qquad + \frac{d - d_\I - d_\III - 2^{\beta(n_\IV)}-2}{2}\tau\Big(\iota^{2(n_\III+n_\IV)}\,\zeta^{n_\I-n_\III}\,\cd[n_\I+n_\III+n_\IV]\Big) \\
\end{align*}
In this case, because $n_\II = 0$, we have
\begin{align*}
    n_\IV &= \ell \\
    n_\I &= n_0 = p - m_0 \\
    n_\III &= n_1 = q - m_1 \\
    n_\I + n_\IV &= n - n_1 = k_1 \\
    n_\III + n_\IV &= n - n_0 = k_0.
\end{align*}
Substituting these, along with $\Delta_0 = d_\I$ and $\Delta_1 = d_\III$, we get the expression in the proposition
with $\epsilon = 1$.

Now consider the case where $\epsilon = 0$, which is equivalent to $n_\II > 0$.
We take the expression for $e(F_\I\dirsum F_\III\dirsum F_\IV)$ above,
multiply by $e(F_\II)$, and use Lemmas~\ref{lemma Q conversion} and~\ref{lem:xiQ}:
\begin{align*}
    e(F) &= e(F_\II)e(F_\I\dirsum F_\III\dirsum F_\IV) \\
    &= \frac{d_\II}{2^{n_\II}} Q^{n_\II} \cdot
    \Biggl( \xi \sum_{j=1}^{n_\IV-1}\rem{\binom{n_\IV}{j}} \cwd^{\;n_\I+j}\cxwd^{\;n_\III+n_\IV-j}\cxwt^{j-1}\cwt^{n_\IV-j-1} \\
    &\qquad + d_\I \cwd^{\;n_\I}\cxwd^{\;n_\III+n_\IV}\cwt^{n_\IV} + d_\III \cwd^{\;n_\I+n_\IV}\cxwd^{\;n_\III}\cxwt^{n_\IV} \\
    &\qquad + \frac{d_\I d_\III d_\IV - d_\I - d_\III - 2^{\beta(n_\IV)}-2}{2}\tau\Big(\iota^{2(n_\III+n_\IV)}\,\zeta^{n_\I-n_\III}\,\cd[n_\I+n_\III+n_\IV]\Big) \Biggr) \\
    &= \frac{d_\II}{2} \tau(\iota^2 c^{n_\II}) \sum_{j=1}^{n_\IV-1}\rem{\binom{n_\IV}{j}} \cwd^{\;n_\I+j}\cxwd^{\;n_\III+n_\IV-j}\cxwt^{j-1}\cwt^{n_\IV-j-1} \\
    &\qquad + \frac{d_\II}{2^{n_\II}}d_\I \cwd^{\;n_\I}\cxwd^{\;n_\III+n_\IV}\cwt^{n_\IV}Q^{n_\II} 
        + \frac{d_\II}{2^{n_\II}}d_\III \cwd^{\;n_\I+n_\IV}\cxwd^{\;n_\III}\cxwt^{n_\IV}Q^{n_\II} \\
    &\qquad + \frac{d - d_\I d_\II - d_\II d_\III - 2^{\beta(n_\IV)}d_\II-2d_\II}{2}\tau\Big(\iota^{2(n_\III+n_\IV)}\,\zeta^{n_\I-n_\III}\,\cd[n]\Big) \\
    &= d_\I d_\II \cwd^{\;n_\I+n_\II}\cxwd^{\;n_\III+n_\IV}\cwt^{n_\IV-n_\II}
        + d_\II d_\III \cwd^{\;n_\I+n_\IV}\cxwd^{\;n_\II+n_\III}\cxwt^{n_\IV-n_\II} \\
    &\qquad + \frac{d - d_\I d_\II - d_\II d_\III}{2}\tau\Big(\iota^{2(n_\III+n_\IV)}\,\zeta^{n_\I-n_\III}\,\cd[n]\Big)
\end{align*}
Rewriting with the notation given in the proposition, we get the expression in the proposition
with $\epsilon = 0$.
\end{proof}

\section{Singular manifolds}\label{sec:singular}

We now want to discuss how the various terms in Propositions~\ref{prop:nIV less nII} and~\ref{prop:nII less nIV}
can be represented by manifolds in a way analogous to how varieties are used
to define the Chow ring. It is well known that manifolds {\it per se} cannot be used to define singular (co)homology,
but we can get around this by allowing limited singularities.

\subsection{Basics}
In \cite{HM:singularities}, Hastings and Waner showed that ordinary $RO(G)$-graded equivariant homology could be
defined using bordism of manifolds with a certain kind of singularity, similar to the Baas-Sullivan approach nonequivariantly.
The Hastings-Waner result can be generalized to show that ordinary $RO(\Pi B)$-graded homology can be given
by bordism of singular manifolds as well. 
We review the definitions and the results we need.
Our context for the moment is a finite group $G$, a $G$-space $B$, and $G$-spaces over $B$.
In the following, we use the notion of equivariant orientation
discussed in \cite{CMW:orientation}.

\begin{definition}
For $\gamma\in RO(\Pi B)$, we define a {\em $\gamma$-oriented manifold over $B$} to be a smooth $G$-manifold $M$
with a map $f\colon M\to B$ together with a specified stable map $\tau_M\to f^*\gamma$ of representations
of $\Pi M$, where $\tau_M$ is the tangent representation of $M$.
We say that $M$ has \emph{dimension} $\gamma$ to mean that $M$ is $\gamma$-orientable,
that is, there exists a $\gamma$-orientation $\tau_M\to f^*\gamma$, but we do not pick one.
\end{definition}

Here, stable maps are defined with respect to addition of arbitrary (finite) representations of $G$.

\begin{definition}\label{def:singularmfld}
We wish to define a {\em $\gamma$-dimensional singular manifold of depth $n$ over $B$}, where $n\geq 0$,
which we do recursively on $n$. 
A $\gamma$-dimensional manifold of depth 0 is defined to be a compact $\gamma$-dimensional manifold over $B$,
with possibly empty boundary.
For $n>0$, a singular manifold of depth (at most) $n$ consists of a $G$-space $X$ over $B$ together with a decomposition
$X = N\union_{\bndry_1 N} M\phi$, where 
\begin{enumerate}
\item
$N$ is a compact $\gamma$-dimensional manifold over $B$ with (possibly empty) boundary $\bndry N$;

\item
$\bndry N = \bndry_1 N\union_\bndry \bndry_2 N$ is a decomposition of $\bndry N$ into two
codimension-0 submanifolds sharing a common boundary $\bndry = \bndry\bndry_1 N = \bndry\bndry_2 N$;

\item
$M\phi$ is the mapping cylinder of a $G$-map
$\phi\colon (\bndry_1 N,\bndry\bndry_1 N) \to (P,\bndry P)$,
where $(P,\bndry P)$ is a disjoint union of singular manifolds of depth no more than $n-1$,
each of which has dimension $\gamma-k$ for some integer $k\geq 0$.
We call $P$ the {\em singular part} of $X$.
\end{enumerate}
Recursively, we define the boundary of $X$ to be the $(\gamma-1)$-dimensional singular manifold
\[
 \bndry X = \bndry_2 N \union_{\bndry\bndry_2 N} M(\bndry\phi)
\]
where $\bndry\phi\colon (\bndry\bndry_2N,\emptyset) = (\bndry\bndry_1N,\emptyset) \to (\bndry P,\emptyset)$ is the restriction of $\phi$.
We orient $\bndry_2 N$ using the identification $\R\dirsum\tau_{\bndry_2 N} = \tau_N|\bndry_2 N$ with
$\R$ the inward normal.

If we wish to specify the singular parts of $X$ in our notation, we will write
$(X;P_1,P_2,\ldots,P_n)$ where $P_1$ is the singular part of $X$, $P_2$ is the singular part of $P_1$, and so on.

We say that a $\gamma$-dimensional singular manifold $N\union_{\bndry_1 N} M\phi$ is
\emph{$\gamma$-oriented} if we have specified a $\gamma$-orientation of $(N,\bndry N)$.
\end{definition}

We use the phrase ``singular manifold,'' without specifying the depth,
to mean a singular manifold of any finite depth.

\begin{definition}\label{def:codim2singularities}
A $\gamma$-dimensional singular manifold $X$ {\em has codimension-2 singularities} if its singular
part $P$ has components of dimensions each no larger than $\gamma-2$.
\end{definition}

Note that we put no further restrictions on $P$---in particular, its own singularities can
be of any non-negative codimensions.

A simple but important observation is the following.

\begin{lemma}\label{lem:homiso}
If $X = N\union_{\bndry_1 N} M\phi$ is a $\gamma$-oriented singular manifold with co\-di\-men\-sion-2 singularities, then
\begin{align*}
 H^G_\gamma(X_+) &\iso H^G_\gamma(N,\bndry_1 N)
\intertext{and}
 H^G_\gamma(X,\bndry X) &\iso H^G_\gamma (N,\bndry N).
\end{align*}
\end{lemma}

\begin{proof}
Consider the following part of the long exact sequence of the pair $(X,M\phi)$:
\[
 H^G_\gamma(M\phi_+) \to H^G_\gamma(X_+) \to H^G_\gamma(X,M\phi) \to H^G_{\gamma-1}(M\phi_+).
\]
Because $M\phi$ is homotopy equivalent to $P$, which is a cell complex of dimension no more than $\gamma-2$,
the first and last terms are 0. Therefore the middle two terms are isomorphic.
By excision, $H^G_\gamma(X,M\phi) \iso H^G_\gamma(N,\bndry_1 N)$, so the first isomorphism follows.

The second isomorphism follows similarly from the long exact sequence of the triple
$(X,\bndry X\union M\phi,\bndry X)$.
\end{proof}

Note that $(N,\bndry N)$ is a smooth $\gamma$-oriented manifold with boundary,
hence has a fundamental class $[N,\bndry N]\in H^G_\gamma(N,\bndry N)$
(\cite[3.11.8]{CostenobleWanerBook}).

\begin{definition}\label{def:fundclass}
If $X = N\union_{\bndry_1 N} M\phi$ is a $\gamma$-oriented singular manifold with co\-di\-men\-sion-2 singularities,
we define $[X,\bndry X]\in H^G_\gamma(X,\bndry X)$ to be the image of the
fundamental class $[N,\bndry N]$ under the isomorphism of Lemma~\ref{lem:homiso}.
If we want to specify the singular parts, we will write
$[X,\bndry X; P_1, P_2, \ldots]$, or $[X;P_1,P_2,\ldots]$ when the boundary is empty.
\end{definition}

Up to this point we have been careful to use the letter $X$ to denote a singular manifold, to prevent confusion
with the $M$ used to denote a mapping cylinder. We will now start to use $M$ rather than $X$
to denote a general singular manifold.

Given any pair of spaces $(X,Y)$ over $B$, a $\gamma$-oriented singular manifold over $(X,Y)$
is a map $f\colon (M,\bndry M)\to (X,Y)$ over $B$ with $(M,\bndry M)$
being a $\gamma$-oriented singular manifold with codimension-2 singularities.
We can then define oriented bordism of such singular manifolds in the usual way, using appropriate singular bordisms,
again with codimension-2 singularities.

More generally, we define a {\em stable} $\gamma$-oriented singular manifold over $(X,Y)$ to be
a $(\gamma+V)$-oriented singular manifold $(M,\bndry M) \to (X,Y)\times (D(V),S(V))$ over $B$
for some $G$-representation $V$.
We call these stable manifolds because we identify such a manifold with its ``suspension''
\[
    (M,\bndry M)\times (D(W),S(W)) \to (X,Y)\times (D(V\dirsum W), S(V\dirsum W))
\]
for any representation $W$.
When $Y=\emptyset$, it is often useful to think of a stable manifold over $X$ as given
by two maps, a map $(M,\bndry M)\to (D(V),S(V))$ (which we think of as the stable manifold) and
a map $M\to X$.
Write $\bar\Omega^G_{RO(\Pi B)}(X,Y)$ for the resulting stable, oriented singular bordism groups of $(X,Y)$.
The proof given by Hastings and Waner generalizes in a straightforward way to show the following,
in which $H^G_{RO(\Pi B)}(X,Y)$ denotes ordinary homology with Burnside ring coefficients.

\begin{theorem}[cf. \protect{\cite[Theorem A]{HM:singularities}}] \label{thm:singularbordism}
The map 
\[
 \bar\Omega^G_{RO(\Pi B)}(X,Y)\to H^G_{RO(\Pi B)}(X,Y)
\]
taking a $(\gamma+V)$-oriented stable manifold
$f\colon (M,\bndry M)\to (X,Y)\times (D(V),S(V))$ to 
\[
    f_*[M,\bndry M] \in H^G_{\gamma+V}((X,Y)\times(D(V),S(V))) \iso H^G_{\gamma}(X,Y)
\]
is a natural isomorphism of homology theories.
\qed
\end{theorem}

When $x\in H^G_{\gamma}(X,Y)$ and $x = f_*[M,\bndry M]$ for a 
(stable, oriented) singular manifold $f\colon (M,\bndry M) \to (X,Y)$,
we say that $x$ is {\em represented by} $(M,\bndry M)$.
When $f$ is understood, we will write $[M,\bndry M]_*$ for $f_*[M,\bndry M]$.

Note that, if $h\colon (X,Y)\to (X',Y')$ is a $G$-map and $x\in H^G_{RO(\Pi B)}(X,Y)$
is represented as $x = f_*[M,\bndry M]$, then
$h_*x = (hf)_*[M,\bndry M]$. That is, the map $h_*$ is given simply by composition with representative
maps $f\colon M\to X$.

The following results give some useful ways to think about singular manifolds and the elements they represent.
Suppose that $(X,Y)$ is a pair of $G$-CW($\gamma$) complexes, with $X^{\gamma+k}$ being the skeleta
of $X$.
We then have that
\[
 H_\gamma^G(X,Y) \iso H_\gamma^G(X,X^{\gamma-2}\union Y),
\]
by looking at the long exact sequence of the evident triple.

\begin{lemma}\label{lem:relative}
Let $(X,Y)$ be a pair of $G$-CW($\gamma$) complexes and
let $M$ represent an element of $H_\gamma^G(X,Y)$, so we have a map
\[
 f\colon (M,\bndry M)\to ((X,Y)\times (D(V),S(V)))
\]
for some $V$.
If $M = N\union_{\bndry_1 N} M\phi$, where $N$ is a $(\gamma+V)$-oriented manifold,
then we may assume that $\bndry N$ maps to 
$(X\times D(V))^{\gamma+V-2}\union (X\times S(V))\union (Y\times D(V))$, and then
the homology class represented by $M$ is also represented by the map
\[
 f|N\colon (N,\bndry N)\to (X\times D(V), (X\times D(V))^{\gamma+V-2}\union (X\times S(V))\union (Y\times D(V))).
\]
\end{lemma}

\begin{proof}
Because $P$ is $(\gamma+V-2)$-dimensional, the restriction $f|P$ is homotopic to a map
into $(X\times D(V))^{\gamma+V-2}$, with $\bndry P$ mapping to $X\times S(V)\union Y\times D(V)$
throughout the homotopy. We can extend to a homotopy of $f$ to a map taking
$\bndry_1 N$ into $(X\times D(V))^{\gamma+V-2}$ while $\bndry_2 N$ still maps
into $X\times S(V)\union Y\times D(V)$.
This homotopy does not change the represented homology class (and can itself
be viewed as a bordism). So we may assume that $f$ takes $\bndry N$
to $(X\times D(V))^{\gamma+V-2}\union X\times S(V)\union Y\times D(V)$.

Now $W = M\times[0,1/2] \union N\times [1/2,1]$ can be viewed as a bordism from
$f$ to $f|N$ over $(X\times D(V), (X\times D(V))^{\gamma+V-2}\union X\times S(V)\union Y\times D(V))$, 
showing that they represent the same homology element in
\begin{multline*}
 H_{\gamma+V}^G(X\times D(V), (X\times D(V))^{\gamma+V-2}\union (X\times S(V))\union (Y\times D(V))) \\
  \iso H_{\gamma+V}^G((X,Y)\times (D(V),S(V)))
  \iso H_\gamma^G(X,Y),
\end{multline*}
proving the lemma.
\end{proof}

The gist of this lemma is that the homology class represented by a singular manifold
with codimension-2 singularities
is essentially determined by the nonsingular part.
This is the idea behind the following two lemmas.

For the first result, recall that an empty manifold can be considered to have any dimension.

\begin{lemma}\label{lem:empty}
Let $(P,\bndry P)$ be a $(\gamma-k)$-oriented singular manifold over $(X,Y)$, 
with $k\geq 2$, and let $(M,\bndry M)$ be the $\gamma$-oriented singular
manifold $M = \emptyset\union_\emptyset M\phi = P$ for $\phi\colon\emptyset\to P$.
Then $(M,\bndry M)$ is cobordant to the empty manifold, hence represents $0$ in homology.
\end{lemma}

\begin{proof}
The restriction of $f\colon M\to X$ to $M \smallsetminus P = \emptyset$ agrees with the similar
restriction of the map $\emptyset\to X$, so the result is implied by Lemma~\ref{lem:relative}.
\end{proof}

\begin{lemma}\label{lem:nonsingular}
Let $(M,\bndry M; P)$ be a $\gamma$-oriented manifold with codimension-2 singularities over $(X,Y)$,
such that $(M,\bndry M)$ is itself a smooth $\gamma$-dimensional manifold.
Then $(M,\bndry M;P)$ and $(M;\bndry M)$ represent the same element in $H_\gamma^G(X,Y)$.
\end{lemma}

\begin{proof}
Write $M = N\union_{\bndry_1 N} M\phi$ as a singular manifold.
Then Lemma~\ref{lem:homiso} tells us that $H_\gamma^G(N,\bndry N) \iso H_\gamma^G(M,\bndry M)$,
from which the current lemma follows.
\end{proof}

We will use this primarily in the nonequivariant case, where it says that considering a codimension-2 submanifold
as singular part doesn't change the homology class.
Equivariantly, there will be interesting cases where we start with a smooth manifold $M$ and
take a submanifold $P$ as singular part, such that we may consider $M\smallsetminus P$ to have a different 
equivariant dimension than $M$. In that case, $M$ and $(M;P)$ will represent 
two different equivariant homology elements, in two different gradings, that restrict to the same
nonequivariant homology element.
See \S\ref{subsub:xi} for an example.

Finally, here is a useful general result that crystallizes the idea that it is the
nonsingular part of a singular manifold that determines what it represents in homology.

\begin{lemma}\label{lem:replacement}
Let $(M_1,\bndry M_1)$ and $(M_2,\bndry M_2)$ be two stable $\gamma$-oriented manifolds with
codimension-2 singularities over $(X,Y)$ related as follows:
Suppose $M_i = N\union_{\bndry_1 N} M\phi_i$ for $i = 1$ and $2$, for the same smooth
$\gamma$-oriented manifold $N$ over $(X,Y)$.
Then $[M_1,\bndry M_1]_* = [M_2,\bndry M_2]_* \in H_\gamma^\GG(X,Y)$.
\end{lemma}

\begin{proof}
This follows directly from Lemma~\ref{lem:relative}.
\end{proof}

Although we have been talking about representing homology classes, our real interest is in representing
cohomology classes. 

\begin{definition}\label{def:cohomrep}
If $X$ is a smooth closed $G$-manifold of dimension $\tau$, we say that an element $x\in H_G^\gamma(X_+)$
is {\em represented by} $f\colon M\to X$ if $M$ is a $(\tau-\gamma)$-oriented singular manifold with codimension-2 singularities
such that $f_*[M] = x\cap [X] \in H^G_{\tau-\gamma}(X_+)$. In other words, $f_*[M]$ is the Poincar\'e dual of $x$.
We write $x = [M]^*$ when $f$ is understood.
As usual, if we want to specify the singular part, we will write $[M;P_1,P_2,\ldots]^*$.
\end{definition} 

We can also interpret this in terms of pushforwards.

\begin{definition}
Given a map $f\colon M\to X$, with $M$ a $\tau_M$-dimensional closed singular $G$-manifold with codimension-2 singularities
and $X$ a $\tau_X$-dimensional closed smooth $G$-manifold, the pushforward map 
\[
 f_!\colon H_G^{\tau_M-\alpha}(M_+)\to H_G^{\tau_X-\alpha}(X_+)
\]
is defined for $\alpha\in RO(\Pi X)$ by the following diagram:
\[
 \xymatrix{
  H_G^{\tau_M-\alpha}(M_+) \ar[r]^{f_!} \ar[d]_{-\cap[M]}
    & H_G^{\tau_X-\alpha}(X_+) \ar[d]^{-\cap[X]}_{\iso} \\
  H^G_{\alpha}(M_+) \ar[r]_{f_*} & H^G_{\alpha}(X_+)
 }
\]
where the rightmost vertical map is the Poincar\'e duality isomorphism.
\end{definition}

Definition~\ref{def:cohomrep} then says that $[M]^* = f_!(1)$ where $1\in H_G^0(M_+)$ is the identity.
Note that, in particular, we have $1 = [X]^* \in H_G^0(X_+)$ using
the identity map $X\to X$.

If $X$ and $Y$ are closed smooth manifolds and $h\colon X\to Y$, then the pushforward map
$h_!\colon H_G^{\tau_X - \alpha}(X_+) \to H_G^{\tau_Y - \alpha}(Y)$
takes an element $x$ represented by $f\colon M\to X$ to $h_! x$ represented by $hf\colon M\to Y$.
It is less easy to describe the map $h^*$ in these terms.

\subsection{Representing elements in the cohomology of a point}\label{sub:cohompoint}
As first examples of how cohomology elements can be represented by singular manifolds, we look at several
elements in the cohomology of a point for $G = \GG$.
(See the appendix for a summary of this cohomology.)
We regard the point as a closed 0-dimensional $\GG$-manifold, so Poincar\'e duality
is the identification $H_\GG^\alpha(S^0) \iso H^\GG_{-\alpha}(S^0)$.
It is easy to see in this case that products in $H_\GG^{RO(\GG)}(S^0)$
correspond to products of the representing manifolds.

\subsubsection{Elements of the zeroth cohomology}\label{subsub:zeroth}
We start by looking at elements in
\[
 H_\GG^0(S^0) \iso A(\GG),
\]
the Burnside ring.
These are represented by 0-dimensional manifolds, that is to say, finite $\GG$-sets.
There are two obvious cases: We can have a single fixed point, which represents the identity, 1,
or we could have a free orbit $\GG/e$ representing the class $g = [\GG/e]\in A(\GG)$.
But these both assume the usual positive orientation.
Putting a negative orientation on the free orbit represents $-g$.

Orientations of a fixed point are more interesting,
because there are four stable orientations of any representation of $\GG$,
determined by the orientation of the trivial part and the nontrivial part of the (virtual) representation.
We analyze the remaining three cases using the fact that elements of $A(\GG)$ are determined by their
nonequivariant restrictions and their fixed sets.
Consider a point with the orientation that is positive on the trivial part and negative on the nontrivial part.
This gives a negative nonequivariant orientation but a positive orientation on taking fixed sets.
The element of the Burnside ring this determines is $1-g$, thus a point with this orientation represents $1-g$.
Similarly, a point with the orientation that is negative on the trivial part and positive on the nontrivial part
represents $-1$, while a point with negative orientations on both parts represents $g-1$.
Note that the units in $A(\GG)$ are exactly the four elements $\pm 1, \pm(1-g)$,
each of which squares to 1.

Put another way: For a point, changing the orientation of the trivial part multiplies by $-1$
while changing the orientation of the nontrivial part multiplies by $1-g$.

\subsubsection{The element $\iota$}\label{sec:iota}
Write $\R^\sigma$ for the nontrivial 1-dimensional real representation of $\GG$.
Consider the element
\[
 \iota \in \Mackey H_{\GG}^{\sigma-1}(S^0)(\GG/e) \iso H^{\sigma-1}(S^0;\ZZ) \iso H_{\GG}^{\sigma-1}(\GG/e_+).
\]
Here is the only place, other than the appendix, where we explicitly think of cohomology as Mackey functor-valued,
so that $\iota$ is an element of that Mackey functor evaluated at the free orbit $\GG/e$.
Via the indicated isomorphisms, we also think of $\iota$ as 
an element of nonequivariant cohomology graded on $RO(\GG)$ via the dimension
homomorphism $RO(\GG)\to\ZZ$, and
as an element of the equivariant cohomology of $\GG/e_+$.
As an element of nonequivariant cohomology, it is represented by the
virtual manifold given by the
identity function $D(\R)\to D(\R^\sigma)$, with the representations there to
make the grading clear.
As an element of the equivariant cohomology of $\GG/e_+$, it
is represented by the $(1-\sigma)$-dimensional stable manifold
\[
 \GG\times (D(\R), S(\R)) \to \GG \times (D(\R^\sigma), S(\R^\sigma))
\]
induced by the nonequivariant identification $D(\R)\to D(\R^\sigma)$,
where we give $\GG\times D(\R)$ the canonical (positive) 1-orientation.

\subsubsection{The element $\xi$}\label{subsub:xi}
The class
\[
 \xi \in H_\GG^{2\sigma-2}(S^0)
\]
can be represented by the $(2-2\sigma)$-dimensional stable singular manifold represented by the identity map $D(\R^{2\sigma})\to D(\R^{2\sigma})$,
where we give the source $M = D(\R^{2\sigma})$ the following
structure as a $2$-oriented singular manifold:
Let $U\subset D(\R^{2\sigma})$ be the open disc around the origin of radius $1/2$
and let $N = M - U$, an annulus. Using the fact that $G$ acts freely on $N$, we can give it
a $2$-orientation. (Note that we cannot do this on all of $M$, because of the fixed point at the origin,
where the tangent space is nontrivial.) Let $\bndry_1 N$ be the inner boundary and $\bndry_2 N$ be the outer
one, let $P$ be a point, considered as a $2-2 = 0$-dimensional manifold, and let $\phi\colon \bndry_1 N\to P$ be the projection. Then
$M$ is homeomorphic to $N\union_{\bndry_1 N} M\phi$ and this structure makes it a $2$-oriented singular manifold.
To emphasize this singular structure, we will use the notation introduced earlier, $M = (D(\R^{2\sigma});0)$,
where $0$ denotes the origin considered as the singular part of $D(\R^{2\sigma})$.

To see that $(D(\R^{2\sigma});0)\to D(\R^{2\sigma})$ represents $\xi$, we use that $\xi$
is characterized by the fact that
$\rho(\xi) = \iota^2$.
To apply $\rho$ to $(D(\R^{2\sigma});0)\to D(\R^{2\sigma})$, we simply forget the action of $\GG$, while remembering
the singular structure. Now we are in a situation where we can apply Lemma~\ref{lem:nonsingular}
(with the trivial group acting),
to say that the nonequivariant map $(D(\R^{2\sigma});0)\to D(\R^{2\sigma})$ represents the same element as
the identity map $D(\R^2)\to D(\R^{2\sigma})$,
which represents $\iota^2$; see \S\ref{sec:iota} above.

We can also use this representation to see that $\xi^\GG = 0$: $(D(\R^{2\sigma});0)^\GG = (0;0)$ 
as a (nonequivariant) $2$-oriented singular manifold and we can apply
Lemma~\ref{lem:empty} to conclude that $(0;0)$ represents $0$.

We could try something analogous to represent an element in $H_\GG^{2-2\sigma}(S^0)$ inverse to $\xi$,
but it will not work:
Starting with the identity $D(\R^2)\to D(\R^2)$, we could try to make the source $D(\R^2)$ into a singular
$2\sigma$-oriented manifold. We can give the annulus $N$ a $2\sigma$-orientation, but $P = 0$ would have to be
considered as a $(2\sigma-2)$-dimensional manifold, which is not possible.
This helps explain the fact that $\xi$ is not invertible in $H_\GG^{RO(G)}(S^0)$.

Positive powers $\xi^n$ can be given a similar description, using $D(\R^{2n\sigma})$ with
singular part the origin.

\subsubsection{The elements $e^n$ and $e^{-n}\kappa$}\label{subsub:e}
For $n > 0$, $e^n$ is represented by the stable $(-n\sigma)$-dimensional manifold
$0\to D(\R^{n\sigma})$.
To see this, we use that $e^n\in H_\GG^{n\sigma}(S^0)$ is characterized by the fact that
$(e^n)^\GG = 1$, and verify that 
\[ 
    [0\to D(\R^{n\sigma})]^\GG = [0\to 0] = 1\in H^0(S^0).
\]

Similarly, for $n > 0$, $e^{-n}\kappa \in H_\GG^{-n\sigma}(S^0)$ is characterized by
$(e^{-n}\kappa)^\GG = 2$. We can represent it by the $\GG$-manifold $S(\R^{1+n\sigma})$,
the unit sphere in $\R^{1+n\sigma}$,
with the following orientation:
Embed $S(\R^{1+n\sigma}) \includesin \R^{1+n\sigma}$ and use as the orientation the identification of the stable tangent bundle
(thought of as the sum of the tangent bundle with the normal bundle, which is trivial) with
$S(\R^{1+n\sigma})\times \R^{1+n\sigma}$.
Nonequivariantly, this gives the usual orientation, but equivariantly it does something interesting.
Call the fixed point $(1,0)\in S(\R\dirsum\R^{n\sigma})$ the north pole and $(-1,0)$ the south pole.
Recalling the discussion of orientations of points in Section~\ref{subsub:zeroth},
the orientation at the north pole is positive in both its trivial and nontrivial parts.
However, at the south pole, if we transport the nonequivariant orientation of the nontrivial part
(that is, the tangent space)
from north to south, we see that the map to $\R^{n\sigma}$ reverses orientation, while on the trivial part
the south pole is oriented in the positive direction. That is, the south pole is oriented positively
in the trivial direction and negatively in the nontrivial direction, the opposite of how we would usually
think about it in terms of an outward normal direction.
Thus, when taking fixed points, the two fixed points embedded in $(\R^{1+n\sigma})^\GG = \R$
will both be oriented in the positive direction, hence their sum represents 2 rather than 0,
making $S(\R^{1+n\sigma})$ with this orientation represent $e^{-n}\kappa$ as claimed.

We can also check from this description that $e^n\cdot e^{-n}\kappa = \kappa$:
The product is represented by the map $S(\R^{1+n\sigma}) \to 0 \includesin D(\R^{n\sigma})$, and there is an obvious homotopy
to a map transverse to the origin with discs of the form $D(\R^{n\sigma})$ at the two poles both
mapping by the identity. The inverse image of the origin is then the two poles,
with the north pole still carrying the orientation that is positive in both the trivial and nontrivial parts
and the south pole oriented in the positive direction on the trivial part and the negative direction on the nontrivial part.
Thus, from our discussion in Section~\ref{subsub:zeroth}, the north pole represents 1 while the south pole represents $1-g$,
so their sum represents $2 - g = \kappa$ as claimed.

This discussion also gives us the case $n = 0$, where we can write $\kappa = [S(\R)]^*$, where
$S(\R)$ is oriented as above.


\subsection{Representing elements in the cohomology of $\Xpq pq$}\label{sub:cohomxpq}
We now turn to the cohomology of finite projective spaces, as summarized in the Appendix.
Throughout we will write $B = \Xpq \infty\infty$ and grade cohomology on $RO(\Pi B)$.

We noted in \cite{CHTFiniteProjSpace} that, if $0\leq p'\leq p$ and $0\leq q'\leq q$, then
\[
    \cwd^{\;p-p'}\cxwd^{\;q-q'} = [\Xpq{p'}{q'}]^* = i_!(1)
\]
where $i\colon \Xpq{p'}{q'}\to \Xpq pq$ is the inclusion. It follows that
\[
    x\cwd^{\;p-p'}\cxwd^{\;q-q'} = i_!i^*(x)
\]
for $x\in H_\GG^{RO(\Pi B)}(\Xpq pq_+)$.
So it suffices in many cases to examine how to represent elements $x$ that are
not multiples of $\cwd$ or $\cxwd$.
In particular, if we want to represent the standard basis given in the Appendix,
then the only thing left to consider is powers of $\cxwt$ and $\cwt$.

\subsubsection{The elements $\cxwt^k$ and $\cwt^k$}\label{subsub:zetas}
For the purpose of representing the terms that appear in 
Propositions~\ref{prop:nIV less nII} and~\ref{prop:nII less nIV},
we need to look at $\cxwt^k$ for $0 < k \leq q$ and at $\cwt^k$ for $0 < k \leq p$.
We continue to use the notation from Definition~\ref{def:singularmfld}.

\begin{proposition}\label{prop:powersOfZeta}
If $0 < k \leq q$ then
\[
 \cxwt^k = [\Xpq pq; \Xpq p{(q-k)}]^*.
\]
If $0 < k \leq p$, then
\[
    \cwt^k = [\Xpq pq; \Xpq {(p-k)}q]^*.
\]
\end{proposition}

\begin{proof}
Consider the case of $\cxwt^k$; the case of $\cwt^k$ is similar.
From \cite{CHTFiniteProjSpace}, we know that $\cxwt^k$ generates
$H_\GG^{k(\chiw-2)}(\Xpq pq_+) \iso A(\GG)$, hence is determined by its restriction to
nonequivariant cohomology,
\[
 \rho(\cxwt^k) = 1 \in H^0(\Xp{p+q}_+)
\]
and its fixed points,
\[
 (\cxwt^k)^\GG = (0,1) \in H^{-2k}(\Xp p_+)\dirsum H^0(\Xq q_+).
\]
Because $\cxwt^k$ lives in grading $k(\chiw-2)$, we need its representing manifold to have this codimension.
The dimension of the ambient $\Xpq pq$ is $p\omega+q\chiw-2$, so the representing manifold needs to have dimension
\[
 p\omega + q\chiw - 2 - k(\chiw - 2) = p\omega + (q-k)\chiw + 2(k-1).
\]
The complement $\Xpq pq \smallsetminus \Xpq p{(q-k)}$ has no fixed points over $\Xp p$ and, over
$\Xq q$, $\chiw = 2$, so we have
\[
 p\omega + q\chiw - 2 = p\omega + (q-k)\chiw + 2(k-1) \qquad\text{away from $\Xp p$},
\]
which is part of what we need. The other thing we need is for the singular part to have integer codimension of at least two, but
the dimension of $\Xpq p{(q-k)}$ is $p\omega+(q-k)\chiw-2$, so its codimension is $2k$, hence it satisfies that requirement.

At this point, we know that $(\Xpq pq; \Xpq p{(q-k)})$ is a singular manifold of the right dimension to represent $\cxwt^k$;
now we have to check that it has the right restriction and fixed points.
Its nonequivariant restriction is the singular manifold $(\Xp{p+q}; \Xp{p+q-k})$.
By Lemma~\ref{lem:nonsingular} (applied with the ambient group being trivial), this represents the same element
as the smooth manifold $\Xp{p+q}$, which represents the identity, as required.
The fixed points are
\[
 (\Xpq pq; \Xpq p{(q-k)})^\GG = (\Xp p; \Xp p) \disjunion (\Xq q; \Xq{(p-k)})
\]
with the first component mapping to $\Xp p$ and the second to $\Xq q$.
Using Lemma~\ref{lem:empty}, the first component represents 0.
By Lemma~\ref{lem:nonsingular}, the second represents 1.
Hence we have verified that we have the correct restriction and fixed points, so
$\cxwt^k = [\Xpq pq; \Xpq p{(q-k)}]^*$ as claimed.
\end{proof}

From the equation $x\cwd^{\;p-p'}\cxwd^{\;q-q'} = i_!i^*(x)$ we noted earlier, we get the following.

\begin{corollary}\label{cor:powersOfZeta}
If $p'\leq p$ and $0 < k \leq q'\leq q$, then
\[
    \cxwt^k\cwd^{\;p-p'}\cxwd^{\;q-q'} = [\Xpq{p'}{q'}; \Xpq {p'}{(q'-k)}]^*.
\]
If $0 < k \leq p'\leq p$ and $q'\leq q$, then
\[
    \cwt^k\cwd^{\;p-p'}\cxwd^{\;q-q'} = [\Xpq{p'}{q'}; \Xpq {(p'-k)}{q'}]^*.
\]
\qed
\end{corollary}

\subsubsection{The elements $\cxwt^i\cwt^j$}
We also need to know how to represent $\cxwt^i\cwt^j$ when $1\leq i\leq q$ and $1\leq j \leq p$.
This product lives in grading $(j-i)\omega -2j + 2i\sigma$ and is characterized by the fact that
$\rho(\cxwt^i\cwt^j) = 1$ and $(\cxwt^i\cwt^j)^\GG = (0,0)$. 
The following is then proved in much the same way as Proposition~\ref{prop:powersOfZeta}.

\begin{proposition}
If $1\leq i\leq q$ and $1\leq j \leq p$, then
\[
    \cxwt^i\cwt^j = [\Xpq pq; \Xpq{p-j}q \union \Xpq p{(q-i)}; \Xpq{p-j}{(q-i)}]^*.
\]
\qed
\end{proposition}

This implies the following.

\begin{corollary}\label{cor:cs times zetas}
If $1\leq i\leq q'$ and $1\leq j \leq p'$, then
\[
    \cxwt^i\cwt^j\cwd^{\;p-p'}\cxwd^{\;q-q'} =
    [\Xpq {p'}{q'}; \Xpq{p'-j}{q'} \union \Xpq {p'}{(q'-i)}; \Xpq{p'-j}{(q'-i)}]^*
\]
\qed
\end{corollary}

\subsubsection{The elements $\cxwt^{-k}\cwd^{\;p}\cxwd^{\;q-q'}$ and $\cwt^{-k}\cwd^{\;p-p'}\cxwd^{\;q}$}
\label{subsub:divided}
In \cite{CHTFiniteProjSpace}, we defined the element $\cxwt^{-k}\cwd^{\;p}\cxwd^{\;q-q'} \in H_\GG^{RO(\Pi B)}(\Xpq pq_+)$
to be the pushforward of $\cxwt^{-k}$ from the group
$H_\GG^{-k(\chiw-2)}(\Xq{q'}_+)$, where $\cxwt$ is invertible.
In fact, because $\chiw = 2$ on $\Xq {q'}$, $\cxwt^{-k} = [\Xq{q'}]^*_{-k(\chiw-2)}$,
where the subscript indicates that we consider $\Xq{q'}$ to have codimension $-k(\chiw-2)$
in itself.
Taking the pushforward, it follows that
\[
 \cxwt^{-k}\cwd^{\;p}\cxwd^{\;q-q'} = [\Xq{q'}]^*_{p\omega+(q-q')\chiw-k(\chiw-2)}.
\]
Again, we can interpet the codimension of $\Xq{q'}$ in $\Xpq pq$ to be as indicated 
using the fact that $\chiw = 2$ on $\Xq{q'}$.
Similarly,
\[
 \cwt^{-k}\cwd^{\;p-p'}\cxwd^{\;q} = [\Xp{p'}]^*_{(p-p')\omega + q\chiw - k(\omega-2)}
\]
using the fact that $\omega = 2$ on $\Xp{p'}$.

Note that these expressions work for any $k\in\ZZ$, giving simpler representatives for
the special cases $\cxwt^{k}\cwd^{\;p}\cxwd^{\;q-q'}$ and $\cwt^{k}\cwd^{\;p-p'}\cxwd^{\;q}$
of the elements considered in Corollary~\ref{cor:powersOfZeta}.
For example, if $0 < k \leq q'$, then we have the two representations
\[
    \cxwt^{k}\cwd^{\;p}\cxwd^{\;q-q'} = [\Xq{q'}; \Xq{(q'-k)}]^* = [\Xq{q'}]^*.
\]
The equality of the two representations also follows from Lemma~\ref{lem:nonsingular}.

\section{Schubert varieties}\label{sec:Schubert}

In this section we present some notations commonly used for Schubert varieties of Grassmannians and explain how these can be generalised to the equivariant world.
We will use them only in the special case of projective spaces, but we suspect they
will be useful more generally.

Let $\Gr(D,N)$ be the Grassmannian of $D$-dimensional $\C$-linear spaces in $\C^N$. Inside of each Grassmannian there exists an important family of subvarieties known as Schubert varieties, which are usually indexed by partitions. In general a  partition 
$\underline{\lambda} = (\lambda_1,\lambda_2,\ldots,\lambda_l)$ of length $l$ is an $l$-tuple  of positive integers 
\[
    \lambda_1\geq \lambda_2\geq \cdots \geq \lambda_l> 0\,,
\]
the \emph{parts} of the partition.
To specify a Schubert variety of $\Gr(D,N)$, a partition must satisfy two extra conditions, on its parts and its length: 
\begin{align*}
\lambda_1\leq N-D\, \quad\text{and }\quad
l\leq D\,.
\end{align*}
Note that we allow the empty partition, of length 0.

Besides a partition, we must also choose a reference flag 
\[
    V_\bullet=\Big(0=V_0\subset V_1 \subset V_2 \subset \cdots\subset V_{N-1} \subset V_N=\C^N\Big)
\]
with $\dim_\C V_i=i$.
We then define the  Schubert variety
\[
    Y_{\underline{\lambda}}(V_\bullet):=\Bigl\{W\in \Gr(D,N)\ \Bigm|\ \dim_\C (W\cap V_{N-D+i-\lambda_i})\geq i, \ \forall\,i \Bigr\}\,.
\]
The sum of the parts of a partition is the codimension of the associated Schubert variety, in other words, one has 
\[
    |\underline{\lambda}|:=\sum_{i=1}^l \lambda_i =\text{codim}\Big(Y_{\underline{\lambda}}(V_\bullet)\,,\Gr(D,N)\Big)\,.
\]
Nonequivariantly, although Schubert varieties do depend on the choice of $V_\bullet$, their fundamental classes are independent of it and, as a consequence, the flag is usually omitted from the notation. 
Equivariantly, different flags can give significantly different Schubert varieties,
so we will suppress $V_\bullet$ from the notation only when it is well understood.

In the special case in which $D=1$ this picture becomes much simpler, with the Grassmannian being $\PP^{N-1}$ and non-empty partitions $\underline{\lambda}$ reducing to single integers $\lambda_1$ such that 
\[
    0< \lambda_1\leq N-1.
\]
In this case, a point $W$ of the Schubert variety is a complex line and the defining condition  implies that  $Y_{\underline{\lambda}}$ is just the smaller projective space 
\[
    \PP^{N-1-\lambda_1}:=\PP(V_{N-\lambda_1}).
\]
(When $\underline\lambda$ is empty, $Y_{\underline\lambda} = \PP^{N-1}$ is the whole projective space.)

So, nonequivariantly, the Schubert varieties in a projective space $\PP^m$ are very simple: they are
sub-projective spaces $\PP^k$ for $k\leq m$. Equivariantly, there are three classes of objects
mapping to $\Xpq pq$
we can consider as Schubert varieties, all of which will appear in our version of B\'ezout's theorem.



\subsection{Free Schubert varieties}
The first type of Schubert variety we consider is one induced from a nonequivariant variety.
If $\lambda_1\leq p+q-1$ and 
$$\Xp {p+q-\lambda_1}=Y_{\underline{\lambda}}\to \Xpq pq$$
is a standard nonequivariant  inclusion, then we have the induced equivariant map
\[
    \GG\times Y_{\underline{\lambda}}\to \Xpq pq.
\]
The cohomology element $[\GG\times Y_{\underline{\lambda}}]^*$ does not depend on which inclusion map we start with, as all such
are homotopic, so we will write
\[
    \Freel {\underline{\lambda}}   = \GG\times Y_{\underline{\lambda}}%
\]
generically and consider the element $[\Freel {\underline{\lambda}}]^*$ in the cohomology of $\Xpq pq$.

Because $[\Xp {p+q-\lambda_1}]^*$ can represent the nonequivariant class
$\iota^k\zeta^j \cd[\lambda_1]$ for any $k$ and $j$,
we can write
\[
    [\Freel {\underline{\lambda}}]^* = \tau(\iota^k\zeta^j\cd[\lambda_1]),
\]
interpreting the equivariant dimension of $\Freel {\underline{\lambda}}$ as appropriate.
Thus, the fundamental classes of free Schubert varieties represent the images under the transfer of the additive generators of 
the nonequivariant cohomology of projective space.

In some of the formulas it will be more convenient to refer to these Schubert varieties in terms of dimension as opposed to codimension. 
To do so, assuming that an ambient space $\Xpq pq$ has been fixed, we will use the notation
\[
    \Freep {p+q-\lambda_1}=\Freel {\underline{\lambda}}\,,
\]
where $p+q-\lambda_1$ is the affine dimension.
We will also use the simplified notation
\[
   \Freel{\lambda_1}  =\Freel{(\lambda_1)} .
\]


\subsection{Invariant Schubert varieties}
Next, consider the case of an equivariant flag $V_\bullet$ in $\Cpq pq$, that is, one for which, for every $i\in\{1,\ldots,p+q\}$, we have 
\[
    tV_i=V_i.
\]
Note that we do not require $V_i$ to be fixed pointwise, just that it is mapped to itself by the action. Thus, each $V_i$ is an equivariant complex affine space, so there exist $0\leq p_i\leq p $ and $0\leq q_i\leq  q$, satisfying $p_i+q_i=i$, for which 
\[
    V_i=\Cpq{p_i}{q_i}.
\]
As in the nonequivariant case, we define the Schubert variety associated to the non-empty partition $\underline{\lambda}=(\lambda_1)$ to be
\[
    Y_{\underline{\lambda}}(V_\bullet):=\Bigl\{W\in \Xpq{p}{q}\ \Bigm|\ \dim_\C (W\cap V_{p+q-\lambda_1})\geq 1\Bigr\}=
\PP(V_{p+q-\lambda_1})
\,.
\]
We will refer to these as the \emph{invariant Schubert varieties} and we will write 
\begin{align*}
    \lambda_1^+ = p - p_{p+q-\lambda_1} \\
\intertext{and}
    \lambda_1^- = q - q_{p+q-\lambda_1}
\end{align*}
to denote, respectively, the codimensions of the components of the fixed sets of $Y_{\underline{\lambda}}(V_\bullet)$ in $\Xp{p}$ and $\Xq{q}$.
Note that $\lambda_1=\lambda_1^+ + \lambda_1^-.$ 
 
For the sake of readability we will usually drop the reference to the flag $V_\bullet$, but the reader should be aware that the dependence on the flag is crucial, even at the level of fundamental classes.
In fact, we saw in \S\ref{sub:cohomxpq} that
\begin{align}\label{eqn FUND}
    [\Xpq{p'}{q'}]^* = \cwd^{\;p-p'}\cxwd^{\;q-q'}
\end{align}
and, in Corollaries~\ref{cor:powersOfZeta} and~\ref{cor:cs times zetas},
 we pointed out that we can represent these elements multiplied by (small) powers of $\cxwt$ and $\cwt$
by adding suitable singular parts consisting of smaller invariant Schubert varieties. Equality (\ref{eqn FUND}) can also be rewritten as
\[ [Y_{\underline{\lambda}}(V_\bullet)]^* = \cwd^{\;\lambda_1^+}\cxwd^{\;\lambda_1^-}
\]
and similarly for the products with (small) powers of $\cxwt$ and $\cwt$. 
Because this class depends only on $\lambda_1$, $\lambda_1^+$, and $\lambda_1^-$, 
and assuming that the ambient space $\Xpq{p}{q}$ is understood, we will write $\ISc{\lambda_1}{\lambda_1^+}{\lambda_1^-}$ instead of the more precise $Y_{\underline{\lambda}}(V_\bullet)$.
With this notation,
\[
    \ISc{\lambda_1}{\lambda_1^+}{\lambda_1^-} = \Xpq{p-\lambda_1^+}{(q-\lambda_1^-)}
\]
and
\[
    [\ISc{\lambda_1}{\lambda_1^+}{\lambda_1^-}]^* = \cwd^{\;\lambda_1^+}\cxwd^{\;\lambda_1^-}.
\]

\subsection{Binate Schubert varieties}\label{sub:binateSchubert}
Finally, we consider the case in which the flag $V_\bullet$ is not necessarily invariant under the action of $C_2$. For each level of the flag we can consider the intersection of $V_i$ with its translate
\[
    I(V_i):=V_i \intersect tV_i= \Cpq{p_i}{q_i}\,,
\]
a $\GG$-subspace of $\Cpq pq$.
We no longer need have that $i = p_i + q_i$, but
we have the following relationship between these numbers.

\begin{proposition}
Let $i$, $p_i$, and $q_i$ be nonnegative integers.
There exists a nonequivariant subspace $V_i\subseteq \Cpq pq$ of dimension $i$
such that $I(V_i)=\Cpq{p_i}{q_i}$ if and only if
\begin{equation}
\begin{aligned}\label{eqn:validms}
    p_i + q_i &\leq i \\
    i-q_i &\leq p \\
    \mathllap{\text{and}}\quad i-p_i &\leq q.
\end{aligned}
\end{equation}
\end{proposition}

\begin{proof}
Note that (\ref{eqn:validms}) implies that $i\leq p+q$, $p_i\leq p$, and $q_i\leq q$,
so we do not need to include these as separate assumptions.

First, assume that there exists such a subspace $V_i$.
By definition, we have $\Cpq{p_i}{q_i} \subseteq V_i$, hence we must have $p_i + q_i \leq i$.
Write
\[
    V_i = I(V_i)\dirsum W,
\]
so that $W\intersect tW = 0$. We claim that $W+tW = \Cpq kk$ where $k = \dim W$.
To see this,
note that $W+tW = \Cpq k{k'}$ for some $k$ and $k'$. Suppose that $k > k'$, so $k > \dim W$.
Then $W\intersect \C^k \neq 0$, because $\dim (W+\C^k) \leq k+k' = 2\dim W < \dim W + k$.
But then we would have
\[
    W\intersect tW \supseteq W\intersect \C^k \neq 0,
\]
a contradiction. Similarly, we cannot have $k < k'$.
Now, $\dim W = i - (p_i + q_i)$, so we get
\begin{align*}
    V_i + tV_i &= \Cpq{p_i}{q_i} \dirsum \Cpq{i-(p_i+q_i)}{(i-(p_i+q_i))} \\
    &= \Cpq{i-q_i}{(i-p_i)}.
\end{align*}
Because $V_i + tV_i \subseteq \Cpq pq$,
it follows that we must have $i-q_i \leq p$ and $i-p_i \leq q$.

Conversely, suppose that $i$, $p_i$, and $q_i$ satisfy (\ref{eqn:validms}).
Let
\begin{align*}
    u_j &= (0,\ldots, 0,1,0\ldots,0; 0, \ldots, 0) && 1\leq j \leq p \\
    v_j &= (0,\ldots, 0; 0, \ldots, 0,1,0,\ldots,0) && 1\leq j \leq q
\end{align*}
be the standard basis elements of $\Cpq pq$.
Then we can take $V_i$ to be the subspace of $\Cpq pq$ with basis
\[
    \{ u_j \mid 1\leq j \leq p_i\} \union \{ v_j \mid 1\leq j \leq q_i\} 
    \union \{u_{p_i+j}+v_{q_i+j} \mid 1\leq j \leq i-(p_i+q_i)\},
\]
which is possible because of (\ref{eqn:validms}).
\end{proof}

We define the \textit{binate Schubert variety} specified by $V_\bullet$ and $\underline\lambda$
to be the union of the following two non equivariant Schubert varieties:
\[
    \BScl{\lambda}(V_\bullet):=\IScl{\lambda}(V_\bullet)\cup t\IScl{\lambda}(V_\bullet).
\]
In general such a $C_2$-space will consist of two projective spaces, $\PP(V_{p+q-\lambda_1})$ and $t\PP(V_{p+q-\lambda_1})$, which are swapped by the action and whose common intersection is $\PP(I(V_{p+q-\lambda_1}))$. As in the previous section we can replace the reference to the flag by keeping track of $\lambda_1^+$ and $\lambda_1^-$, defined the same way and again the codimensions of the fixed sets of $\PP(I(V_{p+q-\lambda_1}))$ in the components of the fixed set of $\Xpq{p}{q}$, and we will write $\BSc{\lambda_1}{\lambda_1^+}{\lambda_1^-}$ instead of $\BScl{\lambda}(V_\bullet)$. 
Now the subscript $\lambda_1$ is not superfluous, as
it is possible to have $\lambda_1^++\lambda_1^- > \lambda_1$.

We would like to able to interpret  binate Schubert varieties as singular manifolds. However, in the dimension dictated
by the singularity having integer codimension, the resulting singular manifold turns out to represent the same
cohomology class as a free Schubert variety.
We can get something more interesting as follows,
which will lead to Corollaries~\ref{cor:Q1} and~\ref{cor:Qk}.


For convenience, and assuming that we have fixed an ambient projective space $\Xpq{p}{q}$, we will now switch to an equivalent notation which highlights the affine dimension of the variety and of its fixed sets: 
\[
    \Spqk{p_i}{q_i}i :=\BSc{p+q-i}{p-p_i}{q-q_i}\,. 
\]
For the purposes of Corollaries~\ref{cor:Q1} and~\ref{cor:Qk}, we would like this to be a singular manifold of (real) equivariant dimension
\begin{equation}\label{eqn:SchubertDim}
    p_i\omega + q_i\chi\omega - 2 + 2(i - p_i - q_i)\sigma
\end{equation}
(which implies a nonequivariant dimension of $2(i-1)$).
The singular part should be the intersection
\[
    \PP(V_i) \intersect t\PP(V_i) =\PP(I(V_i))= \Xpq{p_i}{q_i}.
\]
The action of $\GG$ is free on the complement of the intersection, so we may consider the complement
to have dimension (\ref{eqn:SchubertDim}).
However,  the intersection has dimension
\[
    p_i\omega + q_i\chi\omega - 2,
\]
and the difference of this from (\ref{eqn:SchubertDim}) is $2(i-p_i-q_i)\sigma$
which, when $p_i + q_i \neq i$, is not an integer.
Hence, $(\Spqk{p_i}{q_i}i; \Xpq{p_i}{q_i})$ does
not form a singular manifold with codimension-2 singularities in our desired dimension.
We can fix this by taking a ``desingularization.''
With those corollaries in mind, we are lead to the following definition.

\begin{definition}\label{def:binateSchubert}
Assume that $i$, $p_i$, and $q_i$ are nonnegative and satisfy (\ref{eqn:validms}). Define
\[
    \tBSc{p+q-i}{p-p_i}{q-q_i}=\tSpqk{p_i}{q_i}i = 
    \begin{cases}
        \GG\times \PP^{i-1} \disjunion \Bigl(S(\R^{1+2(i-p_i-q_i)\sigma})\times \Xpq{p_i}{q_i}\Bigr) \\
            \mspace{300mu} \text{if $i > p_i + q_i$} \\
        \Xpq{p_i}{q_i} \disjunion \Xpq{p_i}{q_i} \\
            \mspace{300mu} \text{if $i = p_i + q_i$.}
    \end{cases}
\]
There is a map $\tSpqk{p_i}{q_i}i \to \Xpq pq$ defined as follows.
When $i > p_i+q_i$, we map one copy of $\PP^{i-1}$ onto $\PP(V_i)$, so that the other
maps to $t\PP(V_i)$.
On $S(\R^{1+2(i-p_i-q_i)\sigma})\times \Xpq{p_i}{q_i}$, we take the map to be projection to the second factor followed by inclusion. 
In the case of $i=p_i+q_i$, the map is given by the usual inclusion of $\Xpq{p_i}{q_i}$ on each component.
\end{definition}

Because $S(\R^{1+2(i-p_i-q_i)\sigma})$ has dimension $2(i-p_i-q_i)\sigma$, the space
$\tSpqk{p_i}{q_i}i$ has equivariant dimension (\ref{eqn:SchubertDim}).
By definition, the image of $\tSpqk{p_i}{q_i}i$ in $\Xpq pq$ is the binate Schubert variety
$\Spqk{p_i}{q_i}i$.

Note that, in this definition, we are using the orientation on $S(\R^{1+n\sigma})$
discussed in \S\ref{subsub:e}.
Using that, we could consider the case $i=p_i+q_i$ to be a special case of the first part
of the definition, because, with this orientation, $S(\R)$ represents $\kappa = 2 - g$,
hence the first expression represents the same element as $\Xpq{p_i}{q_i} \disjunion \Xpq{p_i}{q_i}$
in the cohomology of $\Xpq pq$
when $i = p_i+q_i$.

We extend the definition as follows.

\begin{definition}
Assume that $i$ is nonnegative and that $i$, $p_i$, and $q_i$ satisfy (\ref{eqn:validms}),
with $p_i$ and $q_i$ possibly negative. Let
\[
    \tSpqk{p_i}{q_i}i = \tSpqk{\max(p_i,0)}{\max(q_i,0)}i,
\]
but with dimension still given by (\ref{eqn:SchubertDim}).
Similarly, we set
\[
    \tBSc {\lambda_1}{\lambda_1^+}{\lambda_1^-}
        = \tBSc {\lambda_1}{\min(\lambda_1^+,p)}{\min(\lambda_1^-,q)}.
\]
\end{definition}

We need to explain how we can continue to use (\ref{eqn:SchubertDim}) as the dimension when
$p_i < 0$ or $q_i < 0$. Suppose, for example, that $p_i < 0$ and $q_i \geq 0$, so we are setting
\begin{align*}
    \tSpqk{p_i}{q_i}i &= \tSpqk{0}{q_i}i \\
    &= \GG\times\PP^{i-1} \disjunion
        \Bigl(S(\R^{1+2(i-q_i)\sigma})\times \Xq{q_i}\Bigr).
\end{align*}
Because $\GG\times\PP^{i-1}$ is free, we may interpret its dimension as any equivariant dimension
whose underlying nonequivariant dimension is $2(i-1)$, as is true for (\ref{eqn:SchubertDim}).
The piece $S(\R^{1+2(i-q_i)\sigma})\times \Xq{q_i}$ maps entirely into $\Xq q$,
where we have $\omega = 2\sigma$ and $\chi\omega = 2$. Now $S(\R^{1+2(i-q_i)\sigma})\times \Xq{q_i}$
has natural dimension $2(q_i-1) + 2(i-q_i)\sigma$, which agrees with (\ref{eqn:SchubertDim})
on substituting there $\omega = 2\sigma$ and $\chi\omega=2$,
showing that $\tSpqk{0}{q_i}i$ has the desired dimension.
A similar argument works for the case where $p_i \geq 0$ and $q_i < 0$.
If both $p_i$ and $q_i$ are negative, then we are looking at $\tSpqk 00i$, which is free,
so again we may interpret it to have the dimension given by (\ref{eqn:SchubertDim}).

Finally, we identify the element $[\tSpqk{p_i}{q_i}i]^*$.

\begin{proposition}\label{prop:schubertRepresents}
In the cohomology of $\Xpq pq$ we have
\begin{align*}
    [\tSpqk{p_i}{q_i}i]^* &= \tau(\iota^{2(q-(i-p_i))}\zeta^{p-p_i-(q-q_i)}\cd[p+q-i]) \\
    &\qquad {}+ \begin{cases}
        e^{-2(i-p_i-q_i)}\kappa\,\cwd^{\;p-p_i}\cxwd^{\;q-q_i}
            & p_i\geq 0 \text{ and } q_i\geq 0 \\
        e^{-2(i-q_i)}\kappa\,\cxwt^{p_i}\,\cwd^{\;p}\cxwd^{\;q-q_i}
            & p_i < 0 \text{ and } q_i\geq 0 \\
        e^{-2(i-p_i)}\kappa\,\cwt^{q_i}\,\cwd^{\;p-p_i}\cxwd^{\;q}
            & p_i \geq 0 \text{ and } q_i < 0 \\
        0 & p_i < 0 \text{ and } q_i < 0.
      \end{cases} \\
    &= \cwd^{\;p-(i-q_i)}\cxwd^{\,q-(i-p_i)}
        \bigl(\tau(\cd[k]) + e^{-2k}\kappa\cwd^{\;k}\cxwd^{\;k}\bigr)
\end{align*}
where $k = i - p_i - q_i$. In particular, if $k=0$, one has
\[
[\tSpqk{p_i}{q_i}i]^*=2[\Xpq{p_i}{q_i}]^*.
\]

In the alternate notation, we have
\begin{align*}
    [\tBSc{\lambda_1}{\lambda_1^+}{\lambda_1^-}]^* &= \tau\big(\iota^{2(\lambda_1+\lambda_1^+)}\zeta^{\lambda_1^+-\lambda_1^-}\cd[\lambda]\big) \\
    &\qquad {}+ \begin{cases}
        e^{-2(\lambda_1^++\lambda_1^--\lambda_1)}\kappa\,\cwd^{\;\lambda_1^+}\cxwd^{\;\lambda_1^-}
            & p\geq \lambda_1^+ \text{ and } q\geq \lambda_1^- \\
        e^{-2(p+\lambda_1^--\lambda_1)}\kappa\,\cxwt^{p-\lambda_1^+}\,\cwd^{\;p}\,\cxwd^{\;\lambda_1^-}
            & p < \lambda_1^+ \text{ and } q\geq \lambda_1^- \\
        e^{-2(q+\lambda_1^+-\lambda_1)}\kappa\,\cwt^{q-\lambda_1^-}\,\cwd^{\;\lambda_1^+}\cxwd^{\;q}
            & p \geq \lambda_1^+ \text{ and } q < \lambda_1^- \\
        0 & p < \lambda_1^+ \text{ and } q < \lambda_1^-.
      \end{cases} \\
    &= \cwd^{\;\lambda_1-\lambda_1^-}\cxwd^{\,\lambda_1-\lambda_1^+}
        \bigl(\tau(\cd[k]) + e^{-2k}\kappa\cwd^{\;k}\cxwd^{\;k}\bigr)
\end{align*}
where $k = \lambda_1^+ +\lambda_1^- - \lambda_1$.
In the special case $k=0$ one has 
\[
[\tBSc{\lambda_1}{\lambda_1^+}{\lambda_1^-}]^*=2[\ISc{\lambda_1}{\lambda_1^+}{\lambda_1^-}]^*.
\]

\end{proposition}

\begin{proof}
We show the first equality, case by case.
Suppose that 
$p_i\geq 0$ and $q_i\geq 0$.
If $p_i + q_i = i$, so $k = 0$, the equality follows from
the second part of Definition~\ref{def:binateSchubert} and the equality
\[ 
   [\Xpq{p_i}{q_i} \disjunion \Xpq{p_i}{q_i}]^* = 2[\Xpq{p_i}{q_i}]^* = 2\cwd^{\;p-p_i}\cxwd^{\;q-q_i}.
\]
Note that $\tau(\cd[0]) + e^0\kappa = g + (2-g) = 2$.

So assume that $p_i + q_i < i$. We then have that
\[
    \tSpqk{p_i}{q_i}i =
    \GG\times \PP^{i-1} \disjunion \Bigl(S(\R^{1+2(i-p_i-q_i)\sigma})\times \Xpq{p_i}{q_i}\Bigr).
\]
Interpreting the dimension of the free part $\GG\times\PP^{i-1}$ to be given by (\ref{eqn:SchubertDim}),
we have
\[
    [\GG\times\PP^{i-1}]^* = \tau(\iota^{2(q-(i-p_i))}\zeta^{p-p_i-(q-q_i)}\cd[p+q-i]),
\]
the factors of $\iota$ and $\zeta$ being what is necessary to put this element in the correct grading.
For the other piece, we have
\[
    [S(\R^{1+2(i-p_i-q_i)\sigma})]^* = e^{-2(i-p_i-q_i)}\kappa
\]
by \S\ref{subsub:e} and
\[
    [\Xpq{p_i}{q_i}]^* = \cwd^{\;p-p_i}\cxwd^{\;q-q_i}
\]
by \S\ref{sub:cohomxpq}. Putting this all together verifies this case of the proposition.

Now consider the case where $p_i < 0$ and $q_i \geq 0$. We are then setting
\[
    \tSpqk{p_i}{q_i}i = \tSpqk{0}{q_i}i
    = \GG\times\PP^{i-1} \disjunion \Bigl(S(\R^{1+2(i-q_i)\sigma})\times \Xq{q_i}\Bigr),
\]
but with the dimension taken to be given by (\ref{eqn:SchubertDim}), as discussed above.
The free part represents the same element as above, by the same argument.
The other part is
\[
    [S(\R^{1+2(i-q_i)\sigma})\times \Xq{q_i}]^* = e^{-2(i-q_i)}\kappa\cxwt^{p_i}\cwd^{\;p}\cxwd^{\;q-q_i},
\]
where the factor of $\cxwt^{p_i}$ is is there to put the element in the correct grading.
Notice that we are allowed to insert this negative power of $\cxwt$ because we have the factor of $\cwd^{\;p}$,
which is infinitely divisible by $\cxwt$.

The argument for the case $p_i\geq 0$ and $q_i < 0$ is similar.

When $p_i < 0$ and $q_i < 0$, we have $\tSpqk{p_i}{q_i}i = \tSpqk 00i$, which is free
because $\Xpq 00$ is empty.

For the last equality, we have first that
\[
    \cwd^{\;p-(i-p_i)}\cxwd^{\;q-(i-p_i)}\tau(\cd[i-p_i-q_i])
        = \tau(\iota^{2(q-(i-p_i))}\zeta^{p-p_i-(q-q_i)}\cd[p+q-i]),
\]
using the facts that
\[
    \rho(\cwd) = \zeta\cd \qquad\text{and}\qquad \rho(\cxwd) = \iota^2\zeta^{-1}\cd.
\]
Now consider the term
\[
    \cwd^{\;p-(i-q_i)}\cxwd^{\;q-(i-p_i)} \cdot e^{-2k}\kappa\cwd^{\;k}\cxwd^{\;k}
     = e^{-2(i-p_i-q_i)}\kappa \cwd^{\;p-p_i}\cxwd^{\;q-q_i}.
\]
If $p_i\geq 0$ and $q_i\geq 0$, this matches the formula we already proved.
If $p_i < 0$ and $q_i\geq 0$, we have
\[
    e^{-2(i-p_i-q_i)}\kappa \cwd^{\;p-p_i}\cxwd^{\;q-q_i}
    = e^{-2(i-q_i)}\kappa \cxwt^{p_i}\cwd^{\;p}\cxwd^{\;q-q_i}
\]
by iterating the calculation
\begin{align*}
    e^{-2k}\kappa \cwd^{\;p+1}
    &= e^{-2k}\kappa\cxwt^{-1}\cwd^{\;p}\cdot \cxwt\cwd \\
    &= e^{-2k}\kappa\cxwt^{-1}\cwd^{\;p}((1-\kappa)\cwt\cxwd + e^2) \\
    &= e^{-2k}\kappa \xi\cxwt^{-2}\cwd^{\;p}\cxwd + e^{-2(k-1)}\kappa \cxwt^{-1}\cwd^{\;p} \\
    &= e^{-2(k-1)}\kappa \cxwt^{-1}\cwd^{\;p}.
\end{align*}
The case $p_i\geq 0$ and $q_i < 0$ is similar, and, in the case where both are negative,
we use that $\cwd^{\;p}\cxwd^{\;q} = 0$.
\end{proof}

\section{Representing elements using Schubert varieties}\label{sec:representing}

Nonequivariantly, it is well known that the Euler class of a smooth bundle is represented
by the zero locus of a section of the bundle transverse to the zero section,
where we use ``represented'' in the sense of Definition~\ref{def:cohomrep}.
The same can be shown equivariantly and can be improved to allow cases where the
section is not transverse, but the zero locus is a singular manifold
with codimension-2 singularities as in Definition~\ref{def:codim2singularities}.
We will not give the proof here, but take this as motivation for the following discussions.
Our main goal is to show how to represent various terms that appear
in Proposition~\ref{prop:nIV less nII}.

\subsection{The Euler class of $\chi O(2)$}
At the beginning of \S\ref{sec:EulerSums} we introduced the Euler class
\[
    \chi Q = e(\chi O(2)) = \cxwt\cwd + \cwt\cxwd.
\]
The following gives two ways to interpret this geometrically.

\begin{proposition}\label{prop:chiQ}
\begin{align*}
    \chi Q &=[\ISc110;\ISc211]^*+[\ISc101;\ISc211]^*\\
    &=[\Xpq{(p-1)}q;\Xpq{(p-1)}{(q-1)}]^* + [\Xpq p{(q-1)};\Xpq{(p-1)}{(q-1)}]^* \\
    &= [\Xpq{(p-1)}q \union \Xpq p{(q-1)}; \Xpq{(p-1)}{(q-1)}]^*.
\end{align*}
\end{proposition}

\begin{proof}
The first two expressions follow from Corollary~\ref{cor:powersOfZeta}.
They equal the third expression by Lemma~\ref{lem:replacement}.
\end{proof}

We can motivate the third expression in the proposition by looking at sections, as suggested above.
Writing the coordinates of a point in $\Xpq pq$ as $[x_1:\cdots:x_p:y_1:\cdots:y_q]$,
the polynomial $x_1y_1$ gives a section of $\chi O(2)$.
Its zero locus is precisely $\Xpq{(p-1)}q \union \Xpq p{(q-1)}$,
which is not a smooth manifold, but if we take $\Xpq{(p-1)}{(q-1)}$, the intersection of these two hyperplanes,
as singular part, we get a singular manifold with codimension-2 singularities.

\subsection{The Euler class of $O(2)$}\label{sub:Q}
Fundamental to the formula in Proposition~\ref{prop:nIV less nII}
is the element
\[
    Q = e(O(2)) = \tau(c) + e^{-2}\kappa \cwd\cxwd.
\]
Proposition~\ref{prop:schubertRepresents} gives us the following as an immediate corollary.

\begin{corollary}\label{cor:Q1}
\[
    Q = [\tSpqk{p-1}{q-1}{p+q-1}]^* = [\tBSc 111]^*.
\]
\qed
\end{corollary}


We can interpret this result in terms of a section of $O(2)$.
Assume for simplicity that $p>0$ and $q>0$.
The section $s = x_1^2 + y_1^2$ has zero locus
$\PP(W_0)\union \PP(W_1)$ where
\begin{align*}
    W_0 &= \{ (x,y) \in \Cpq pq \mid x_1 = i y_1 \} \\
    W_1 &= \{ (x,y) \in \Cpq pq \mid x_1 = -i y_1 \}.
\end{align*}
Using the notation of \S\ref{sub:binateSchubert}, $W_0$ is a model for $\Fpqk{p-1}{q-1}{p+q-1}$
and $W_1 = tW_0 = t\Fpqk{p-1}{q-1}{p+q-1}$, so the zero locus of $s$ can be identified with
$\Spqk{p-1}{q-1}{p+q-1}$.
Our first thought is to consider $\Spqk{p-1}{q-1}{p+q-1}$ as a singular manifold with singular part
$\Xpq{p-1}{(q-1)}$, but, as discussed in \S\ref{sub:binateSchubert}, the dimensions are wrong and this
does not give us a singular manifold with codimension-2 singularities in the correct dimension.

We can correct this problem by considering $\tilde P$, the blowup of $\Xpq pq$ along $\Xpq{p-1}{(q-1)}$.
Explicitly, this can be modeled as the subspace of $\Xpq pq\times \Xpq 1{}$ defined by
\[
    \tilde P = \{ ([x:y], [u:v]) \in \Xpq pq \times \Xpq 1{} \mid x_1 v = y_1 u \}.
\]
The proper transforms of $\PP(W_0)$ and $\PP(W_1)$ are
\begin{align*}
    \tilde \PP(W_0) &= \{ ([x:y],[u:v])\in \tilde X \mid x_1= iy_1 \text{ and } [u,v] = [1:i] \} \\
    \tilde \PP(W_1) &= \{ ([x:y],[u:v])\in \tilde X \mid x_1= -iy_1 \text{ and } [u,v] = [1:-i] \}.
\end{align*}
We then have $\tilde \PP(W_0)$ and $\tilde \PP(W_1)$ disjoint and 
$\tilde \PP(W_0) \union \tilde \PP(W_1) \homeo \GG\times \PP^{p+q-2}$.
The exceptional divisor in $\tilde P$ is $E = \Xpq{p-1}{(q-1)}\times \Xpq 1{}$,
which intersects $\tilde \PP(W_0)$ in $\Xpq{p-1}{(q-1)}\times [1:i]$ and intersects
$\tilde \PP(W_1)$ in $\Xpq{p-1}{(q-1)}\times [1:-i]$.
These intersections are freely interchanged by $\GG$ giving a copy of 
$\GG\times \Xpq{p-1}{(q-1)} = \GG\times\PP^{p+q-3}$.
This allows us to consider
\[
    (\tilde \PP(W_0) \union \tilde \PP(W_1) \union E; \GG\times\PP^{p+q-3})
\]
as a singular manifold with codimension-2 singularities.
We can remove the singularities using Lemma~\ref{lem:nonsingular}, and the result
is exactly $\tSpqk{p-1}{q-1}{p+q-1}$.

Note that this differs from the nonequivariant algebro-geometric analysis, which would allow us
to replace $\PP(W_0)\union \PP(W_1)$ with $\tilde\PP(W_0)\union \tilde\PP(W_1)$, with no contribution from the
exceptional divisor $E$. Equivariantly, and using ordinary cohomology, we cannot neglect
the exceptional divisor as it contributes the term $e^{-2}\kappa \cwd\cxwd$ to $Q$.
Note that $e^{-2}\kappa$ restricts to 0 nonequivariantly, recovering the classical result.

One further thought about this example: We could also consider the section of $O(2)$ given by
the polynomial $\sum_{i=1}^p x_i^2 + \sum_{j=1}^q y_j^2$. Assuming $p+q > 2$, this is a smooth
section with zero locus the Fermat hypersurface, which we could use as a representative of $Q$.
Taking the proper transform in $\tilde P$ and letting the polynomial deform to $x_1^2+y_1^2$,
we can see that the Fermat hypersurface deforms to $\tilde \PP(W_0) \union \tilde \PP(W_1) \union E$.

\subsection{Powers of $Q$}
We also need to interpret $Q^k$ for $2 \leq k < p+q$. 
By Lemma~\ref{lem Q powers}, we can write
\begin{equation}\label{eqn:Qk}
    Q^k = 2^{k-1}(\tau(c^k) + e^{-2k}\kappa\cwd^{\;k}\cxwd^{\;k}).
\end{equation}
Proposition~\ref{prop:schubertRepresents} has the following corollary.

\begin{corollary}\label{cor:Qk}
For $0\leq k < p+q$,
\[
    Q^k = 2^{k-1}[\tSpqk{p-k}{q-k}{p+q-k}]^* = 2^{k-1}[\tBSc kkk]^*.
\]
\qed
\end{corollary}




We don't have a good interpretation of this representation of $Q^k$ in terms of sections of $kO(2)$.
One issue is that the sphere $S(\R^{1+2k})$ that appears would not arise via any
algebro-geometric construction.

\subsection{Powers of $Q$ times $\zeta$}
The remaining terms that need to be interpreted have the form $\cxwt Q^k$ or $\cwt Q^k$ with $1\leq k < p+q$.

\begin{proposition}\label{prop:zeta times Qk}
For $0\leq k < p+q$,
\begin{align*}
    \cxwt Q^k &= 2^{k-1}[\tSpqk{p-k}{q-k}{p+q-k}; \tSpqk{p-k}{q-k-1}{p+q-k-1}]^*
        = 2^{k-1}[\tBSc kkk; \tBSc {k+1}k{k+1}]^* \\[1.5ex]
    \cwt Q^k &= 2^{k-1}[\tSpqk{p-k}{q-k}{p+q-k}; \tSpqk{p-k-1}{q-k}{p+q-k-1}]^*
        = 2^{k-1}[\tBSc kkk; \tBSc {k+1}{k+1}k]^*
\end{align*}
\end{proposition}

\begin{proof}
We look at $\cxwt Q^k$, the case of $\cwt Q^k$ being similar.

We have
\[
    Q^k = 2^{k-1}[\tSpqk{p-k}{q-k}{p+q-k}]^*
\]
from Corollary~\ref{cor:Qk}, so we need to determine $\cxwt[\tSpqk{p-k}{q-k}{p+q-k}]^*$.

If $k = 0$, the result follows from Proposition~\ref{prop:powersOfZeta}, so assume that $k > 0$.
Consider the case $k\leq p$ and $k\leq q$, so
\[
    \tSpqk{p-k}{q-k}{p+q-k} = \GG\times\PP^{p+q-k-1}
        \disjunion \Bigl(S(\R^{1+2k\sigma})\times \Xpq{p-k}{(q-k)}\Bigr).
\]
For the free part, we have
\begin{align*}
    \cxwt[\GG\times\PP^{p+q-k-1}]^*
    &= [\GG\times\PP^{p+q-k-1}]^* \\
    &= [\GG\times\PP^{p+q-k-1}; \GG\times\PP^{p+q-k-2}]^*
\end{align*}
where the first equality involves a reinterpretation of dimension and the second follows
from Lemma~\ref{lem:nonsingular}.
For the other part, we have
\[
    \cxwt[\Xpq{p-k}{(q-k)}]^* = [\Xpq{p-k}{(q-k)}; \Xpq{p-k}{(q-k-1)}]^*
\]
by Corollary~\ref{cor:powersOfZeta}. Putting these representations together shows the proposition
in this case.

The other cases, where $k>p$ or $k>q$, follow similarly.
\end{proof}

We note the following special cases where the representation simplifies.

\begin{corollary}\label{cor:zeta times Qk simpl}
If $k\geq p$, one has
\[
[\tSpqk{p-k}{q-k}{p+q-k}; \tSpqk{p-k}{q-k-1}{p+q-k-1}]^* = [\tSpqk{p-k}{q-k}{p+q-k}]^*,
\]
which implies
\[
    \cxwt Q^k = 2^{k-1}[\tSpqk{p-k}{q-k}{p+q-k}]^* = 2^{k-1}[\tBSc kkk]^*.
\]
Similarly, if $k\geq q$, one has
\[
[\tSpqk{p-k}{q-k}{p+q-k}; \tSpqk{p-k-1}{q-k}{p+q-k-1}]^* = [\tSpqk{p-k}{q-k}{p+q-k}]^*,
\]
which implies
\[
    \cwt Q^k = 2^{k-1}[\tSpqk{p-k}{q-k}{p+q-k}]^* = 2^{k-1}[\tBSc kkk]^*.
\]
\end{corollary}

\begin{proof}
Consider the case $k\geq p$. Then we have
\begin{align*}
    \tSpqk{p-k}{q-k}{p+q-k} &= \GG\times\PP^{p+q-k} \disjunion \Bigl(S(\R^{1+2k\sigma})\times \Xq{(q-k)}\Bigr) \\
\intertext{and}
    \tSpqk{p-k}{q-k-1}{p+q-k-1} &= \GG\times\PP^{p+q-k-1} \disjunion \Bigl(S(\R^{1+2k\sigma})\times \Xq{(q-k-1)}\Bigr)
\end{align*}
The equality
\[
    [\tSpqk{p-k}{q-k}{p+q-k}; \tSpqk{p-k}{q-k-1}{p+q-k-1}]^* = [\tSpqk{p-k}{q-k}{p+q-k}]^*
\]
then follows from Lemma~\ref{lem:nonsingular}. 
Note that this does not work in the case $k < p$ because it would require a reinterpretation of the dimension
of $\tSpqk{p-k}{q-k}{p+q-k}$ that is not possible.
The case $k\geq q$ is similar.
\end{proof}

\section{A geometric $\GG$-B\'ezout theorem}\label{sec:Bezout}

We can now state and prove our equivariant version of B\'ezout's theorem,
using the equivariant Schubert varieties introduced in \S\ref{sec:Schubert}.


\begin{theorem}[B\'ezout's Theorem]\label{thm:Bezout}
Assume Context~\ref{context} and the notation around it.
\begin{enumerate}
    \item If $\ell = m_0 + m_1 - m \leq 0$, let 
    $\Delta_{\min} = \min(\Delta_0,\Delta_1)$ and $\Delta_{\max} = \max(\Delta_0,\Delta_1)$. Then
    \begin{align*}
        e(F) &= \frac{\Delta_{\min}}{2}[\tSpqk{m_0}{m_1}m]^* \\
        &\qquad + \frac{\Delta_0-\Delta_{\min}}{2}[\tSpqk{m_0}{m_1-1}m; \tSpqk{m_0-1}{m_1-1}{m-1}]^* \\
        &\qquad + \frac{\Delta_1-\Delta_{\min}}{2}[\tSpqk{m_0-1}{m_1}m; \tSpqk{m_0-1}{m_1-1}{m-1}]^* \\
        &\qquad + \frac{\Delta - \Delta_{\max}}{2}[\Freep m]^*.
    \end{align*}

    This formula simplifies in the following special cases:
    \[
        e(F) =
        \begin{cases}
            \dfrac{\Delta_1}{2}[\tSpqk{0}{m_1}m]^* + \dfrac{\Delta-\Delta_1}{2}[\Freep m]^*
                &\text{if $m_0 \leq 0$} \\[3ex]
            \dfrac{\Delta_0}{2}[\tSpqk{m_0}{0}m]^* + \dfrac{\Delta-\Delta_0}{2}[\Freep m]^*
                &\text{if $m_1 \leq 0$} \\[3ex]
            \dfrac{\Delta}{2}[\Freep m]^*
                &\text{if $m_0 \leq 0$ and $m_1 \leq 0$.}
        \end{cases}
    \]
    (Here, we need to interpret $\tSpqk{0}{m_1}m$ as having the correct dimension, and
    similarly for $\tSpqk{m_0}{0}m$.)

    \item If $\ell > 0$, let $\epsilon = \overline\Delta_0$. Then
    \begin{align*}
        e(F)
        &= \epsilon\sum_{j=1}^{\ell-1}\rem{\binom{\ell}{j}} [\Xpq {m_0 - j}{(m-m_0 + j)}; \\
        &\mspace{120mu} \Xpq{m-m_1}{(m-m_0+j)}\union \Xpq{m_0-j}{(m-m_0)}; \\
        &\mspace{120mu} \Xpq{m-m_1}{(m-m_0)}]^* \\
        &\qquad + \Delta_0 [\Xpq{m_0}{(m-m_0)}; \Xpq{m-m_1}{(m-m_0)}]^* \\
        &\qquad + \Delta_1 [\Xpq{m-m_1}{m_1}; \Xpq{m-m_1}{(m-m_0)}]^* \\
        &\qquad + \frac{\Delta - \Delta_0 - \Delta_1 - \epsilon(2^{\beta(\ell)}-2)}{2} [\Freep m]^*
    \end{align*}
\end{enumerate}
We can also write the result in terms of codimensions, as follows.
\begin{enumerate}
    \item If $\ell = n - n_0 - n_1 \leq 0$, let 
    $\Delta_{\min} = \min(\Delta_0,\Delta_1)$ and $\Delta_{\max} = \max(\Delta_0,\Delta_1)$. Then
    \begin{align*}
        e(F) &= \frac{\Delta_{\min}}{2}[\tBSc{n}{n_0}{n_1}]^* \\
        &\qquad + \frac{\Delta_0-\Delta_{\min}}{2}[\tBSc{n}{n_0}{n_1+1}; \tBSc{n+1}{n_0+1}{n_1+1}]^* \\
        &\qquad + \frac{\Delta_1-\Delta_{\min}}{2}[\tBSc{n}{n_0+1}{n_1}; \tBSc{n+1}{n_0+1}{n_1+1}]^* \\
        &\qquad + \frac{\Delta - \Delta_{\max}}{2}[\Freel {n}]^*.
    \end{align*}

    This formula simplifies in the following special cases:
    \[
        e(F) =
        \begin{cases}
            \dfrac{\Delta_1}{2}[\tBSc{n}{p}{n_1}]^* + \dfrac{\Delta-\Delta_1}{2}[\Freel {n}]^*
                &\text{if $n_0 \geq p$} \\[3ex]
            \dfrac{\Delta_0}{2}[\tBSc{n}{n_0}{q}]^* + \dfrac{\Delta-\Delta_0}{2}[\Freel {n}]^*
                &\text{if $n_1 \geq q$} \\[3ex]
            \dfrac{\Delta}{2}[\Freel {n}]^*
                &\text{if $n_0 \geq p$ and $n_1 \geq q$.}
        \end{cases}
    \]
    (Here, we need to interpret $\tBSc{n}{p}{n_1}$ as having the correct dimension, and
    similarly for $\tBSc{n}{n_0}{q}$.)

    \item If $\ell > 0$, let $\epsilon = \overline\Delta_0$. Then
    \begin{align*}
        e(F)
        &= \epsilon\sum_{j=1}^{\ell-1}\rem{\binom{\ell}{j}} \Big[\ISc n{n_0+j}{n-n_0-j};\ 
                \ISc {n+\ell-j}{n-n_1}{n-n_0-j}  \union \ISc {n+j}{n_0+j}{n-n_0};\ 
                \ISc {n+\ell}{n-n_1}{n-n_0}\Big]^* \\
        &\qquad + \Delta_0 [\ISc n{n_0}{n-n_0}; \ISc {n+\ell}{n-n_1}{n-n_0}]^*
                + \Delta_1 [\ISc n{n-n_1}{n_1}; \ISc {n+\ell}{n-n_1}{n-n_0}]^* \\
        &\qquad + \frac{\Delta - \Delta_0 - \Delta_1 - \epsilon(2^{\beta(\ell)}-2)}{2} [\Freel {n}]^*.
    \end{align*}
\end{enumerate}
\end{theorem}

\begin{proof}
Suppose first that $m_0 + m_1\leq m$. 
Then $e(F)$ is given by Proposition~\ref{prop:nIV less nII}:
\begin{align*}
    e(F) &= \frac{\Delta_{\min}}{2^{\ell}}\cwd^{\;k_1}\cxwd^{\;k_0} Q^{\ell} 
    + \frac{\Delta_0-\Delta_{\min}}{2^{\ell + 1}} \cwd^{\;k_1-1}\cxwd^{\;k_0} \cwt Q^{\ell+1} \\
    &\qquad {}+ \frac{\Delta_1-\Delta_{\min}}{2^{\ell + 1}} \cwd^{\;k_1}\cxwd^{\;k_0-1} \cxwt Q^{\ell+1} 
    + \frac{\Delta - \Delta_{\max}}{2} \tau(\iota^{2k_0}\zeta^{k_1-k_0} \cd[p+q-m])
\end{align*}
where
\begin{align*}
    k_0 &= q-(m-m_0) \\
    k_1 &= p-(m-m_1) \\
   \mathllap{\text{and}\qquad} \ell &= m - m_0 - m_1.
\end{align*}

By Corollary~\ref{cor:Qk} and Proposition~\ref{prop:zeta times Qk}, we have
\begin{align*}
    \cwd^{\;k_1}\cxwd^{\;k_0} Q^{\ell}
        &= 2^{\ell-1}[\tSpqk{m_0}{m_1}m]^* \\
    \cwd^{\;k_1-1}\cxwd^{\;k_0} \cwt Q^{\ell+1}
        &= 2^{\ell}[\tSpqk{m_0}{m_1-1}m; \tSpqk{m_0-1}{m_1-1}m]^* \\
\intertext{and}
    \cwd^{\;k_1}\cxwd^{\;k_0-1} \cxwt Q^{\ell+1}
        &= 2^{\ell}[\tSpqk{m_0-1}{m_1}m; \tSpqk{m_0-1}{m_1-1}m]^*.
\end{align*}
The first follows from Corollary~\ref{cor:Qk} when we think of $\cwd^{\;k_1}\cxwd^{\;k_0} Q^{\ell}$
as the pushforward of $Q^{\ell}$ from the cohomology of $\Xpq{m-m_1}{(m-m_0)}$,
and similarly for the other two. 
They can also be derived directly from Proposition~\ref{prop:schubertRepresents}.
This gives the general formula in part (1) of the theorem.

If $m_0 \leq 0$, then $\Delta_{\min} = \Delta_0 = 0$. 
(Here, we are implicitly assuming that $\Delta_1 \geq 0$, but a similar simplification
can be given if $\Delta_1 < 0$.)
Together with Corollary~\ref{cor:zeta times Qk simpl},
we get the simplified formula given in part (1). The other simplifications follow similarly.
These simplifications also correspond to the simplified formulas in Proposition~\ref{prop:nIV less nII}.

Now consider the case where $m_0 + m_1 > m$, so that $e(F)$ is given by
Proposition~\ref{prop:nII less nIV}:
\begin{align*}
    e(F) &= \epsilon\xi\sum_{j=1}^{\ell-1}\rem{\binom{\ell}{j}}\cwd^{\;p-m_0+j}\;\cxwd^{\;k_0-j}\,\cxwt^{j-1}\,\cwt^{\;\ell-j-1} \\
    &\qquad {} + \Delta_0 \cwd^{\;p-m_0}\cxwd^{\;k_0}\cwt^{\;\ell}
        + \Delta_1 \cwd^{\;k_1}\cxwd^{\;q-m_1}\cxwt^{\;\ell}   \\
    &\qquad {}+ \frac{\Delta - \Delta_0 - \Delta_1 - \epsilon(2^{\beta(\ell)}-2)}{2}\tau(\iota^{2k_0}\zeta^{k_1-k_0}\cd[p+q-m])
\end{align*}
where $k_0$ and $k_1$ are as above, but now $\ell = m_0 + m_1 - m$.
By Corollary~\ref{cor:cs times zetas}, we have
\begin{align*}
    \xi\cwd^{\;p-m_0+j}\;\cxwd^{\;k_0-j}\,&\cxwt^{j-1}\,\cwt^{\;\ell-j-1}\\
    &= \cwd^{\;p-m_0+j}\;\cxwd^{\;k_0-j}\,\cxwt^{j}\,\cwt^{\;\ell-j} \\
    &= [\Xpq {m_0 - j}{(q- k_0 + j)}; \\
    &\qquad \Xpq{m_0-\ell}{(q-k_0+j)}\union \Xpq{m_0-j}{(q-k_0)}; \\
    &\qquad \Xpq{m_0-\ell}{(q-k_0)}]^* \\
    &= [\Xpq {m_0 - j}{(m - m_0 + j)}; \\
    &\qquad \Xpq{m-m_1}{(m-m_0+j)}\union \Xpq{m_0-j}{(m-m_0)}; \\
    &\qquad \Xpq{m-m_1}{(m-m_0)}]^*
\end{align*}
on substituting the definitions of $k_0$ and $\ell$.
By Corollary~\ref{cor:powersOfZeta}, we have
\begin{align*}
    \cwd^{\;p-m_0}\cxwd^{\;k_0}\cwt^{\ell}
        &= [\Xpq{m_0}{(q-k_0)}; \Xpq{m_0-\ell}{(q-k_0)}]^* \\
        &= [\Xpq{m_0}{(m-m_0)}; \Xpq{m-m_1}{(m-m_0)}]^* \\
    \cwd^{\;k_1}\cxwd^{\;q-m_1}\cxwt^{\;\ell}
        &= [\Xpq{m-m_1}{m_1}; \Xpq{m-m_1}{(m-m_0)}]^*.
\end{align*}
Substituting these gives us the
formula in the theorem.
\end{proof}

To illustrate this result, we look in detail at several special cases.

\subsection{Special case: codimension 1}

This is the case where $n = 1$, meaning that we are looking at the Euler class of a single line bundle $L$. 
In other words, we are giving a geometric interpretation of Proposition~\ref{prop: base case}.

\begin{corollary}\label{cor:codim1}
In the cohomology of $\Xpq pq$, for $k\in\Z$, we have:
\begin{alignat*}{2}
 \textup{i)}&\quad& e(O(2k+1)) &= [\ISc 110]^* + k[\tBSc 111; \tBSc 221]^* \\
 \textup{ii)}&&  e(O(2k)) &= k[\tBSc 111]^* \\
 \textup{iii)}&& e(\chi O(2k+1)) &= [\ISc 101]^* + k[\tBSc 111; \tBSc 212]^* \\
 \textup{iv)}&& e(\chi O(2k)) &= [\ISc 101;\ISc 211]^* 
                     + [\ISc 110; \ISc 211]^* 
                     + (k-1)[\Freel{1}]^*.
\end{alignat*}
\end{corollary}

\begin{proof}






In each case we have $n = 1$. Write $L$ for the line bundle in question.

(i): This is the case $n_0 = 1$ and $n_1 = 0$.
We have $\Delta_0 = 2k+1$, $\Delta_1 = 1$, and
$\Delta = \Delta_0$, so Theorem~\ref{thm:Bezout} gives us
\begin{align*}
    e(L) &= \frac 12[\tBSc 110]^* + \frac{\Delta_0-1}2[\tBSc 111; \tBSc 221]^* \\
    &= [\ISc 110]^* + k[\tBSc 111; \tBSc 221]^*.
\end{align*}
(iii) is similar.

(ii): This is the case $n_0 = 1$ and $n_1 = 1$.
We have $\Delta = \Delta_0 = \Delta_1 = 2k$.
Theorem~\ref{thm:Bezout} gives us
\[
    e(L) = \frac{\Delta}2[\tBSc 111]^* = k[\tBSc 111]^*.
\]

(iv): In the last case, $n_0 = 0$ and $n_1 = 0$.
We have $\Delta_0 = \Delta_1 = 1$ and $\Delta = 2k$.
Theorem~\ref{thm:Bezout} gives us
\begin{align*}
    e(L) &= [\ISc 101; \ISc 211]^* + [\ISc 110; \ISc 211]^* + \frac{\Delta-2}2[\Freel 1]^* \\
    &= [\ISc 101; \ISc 211]^* + [\ISc 110; \ISc 211]^* + (k-1)[\Freel 1]^*. \qedhere
\end{align*}
\end{proof}

\subsection{Special case: dimension 0}

In \cite{CHTFiniteProjSpace}, we gave a 0-dimensional B\'ezout theorem as an application of our calculation
of the cohomology of finite projective spaces. We now show how to recover that result as a consequence of
our general B\'ezout theorem. In stating our result we will use the following spaces, which we here interpret geometrically for the reader's convenience.

\vspace{0.2 cm}
\begin{description}\setlength{\itemsep}{7pt}
\item [$\Xp{}$\,] Stands for a point in $\Xp{p}$, so a point fixed by the $C_2 $-action, whose associated affine line is pointwise fixed.

\item [$\Xq{}$\,] Stands for a point in $\Xq{q}$, so a point fixed by the $C_2 $-action, whose associated affine line is flipped by the action.

\item [$\Freep 1$\,] Stands for a pair of points that are exchanged by the $C_2$-action, so a copy of a free orbit $\GG$. At the affine level, these are two distinct lines that are exchanged by the action of $C_2$.
\end{description}

\begin{corollary}[Theorem~7.12 of \cite{CHTFiniteProjSpace}]
If $n = p+q-1$, so $m=1$, then
\[
    e(F) = \Delta_0[\Xp{}]^* + \Delta_1[\Xq{}]^* + \frac{\Delta-\Delta_0-\Delta_1}{2}[\Freep 1]^*.
\]
\end{corollary}

Thus, $e(F)$ is represented by the finite $\GG$-set mapping into $\Xpq pq$ described as
$\Delta_0$ fixed points in $\Xp p$, $\Delta_1$ fixed points in $\Xq q$,
and $(\Delta-\Delta_0-\Delta_1)/2$ free orbits.
Recall our convention that $\Delta_0 = 0$ if $n_0\geq p$ ($m_0 \leq 0$) and $\Delta_1 = 0$ if $n_1\geq q$ ($m_1\leq 0$).

\begin{proof}
With $m = 1$, we have $m_0\leq 1$ and $m_1\leq 1$.
That means there is only one case where $m_0+m_1 > m$, which is $m_0 = m_1 = 1$.
Theorem~\ref{thm:Bezout} then gives us
\[
e(F) = \Delta_0[\Xp{}]^* + \Delta_1[\Xq{}]^* + \frac{\Delta-\Delta_0-\Delta_1}{2}[\Freep 1]^*
\]
on removing empty spaces.

If $m_0 + m_1 \leq m$, then we will have either $m_0\leq 0$, so $\Delta_0 = 0$, or $m_1\leq 0$, so $\Delta_1 = 0$.
This gives $\Delta_{\min} = 0$.
(As elsewhere, we are implicitly assuming that $\Delta_0 \geq 0$. A similar analysis
can be made if $\Delta_0 < 0$.)
Consider the case $m_0 = 1$ and $m_1\leq 0$.
Then one of the simplified forms in Theorem~\ref{thm:Bezout} gives us
\[
    e(F) = \frac{\Delta_0}2[\tSpqk{1}{m_0-1}1]^* + \frac{\Delta-\Delta_0}2[\Freep 1]^*
\]
Now
\begin{align*}
    [\tSpqk 1{m_0-1}1]^* &= [\tSpqk 101]^* \\
    &= [\Xp{}\disjunion\Xp{}]^* \\
    &= 2[\Xp{}]^*,
\end{align*}
which verifies the formula in the corollary.
The case $m_0\leq 0$ and $m_1 = 1$ is similar.

Finally, the case $m_0 \leq 0$ and $m_1\leq 0$ has $\Delta_0 = \Delta_1 = 0$.
In this case, Theorem~\ref{thm:Bezout} gives us
\[
    e(F) = \frac{\Delta}2[\Freep 1]^*,
\]
completing the verification of the corollary.
\end{proof}

\subsection{Special case: dimension 1}
In this case we have that $n=p+q-2$, so $m=2$, and the different instances of Theorem \ref{thm:Bezout} can be summarized in the following table. 

\begin{center}
\phantom{a}\hspace{-4 cm}
\begin{tabular}{|c|c|c|c|}
\hline
&&&\\
 $m_0$ \textbackslash \  $m_1$   &   2  &   1  &  $\leq\,$0\\
 &&&\\
\hline
 &&&\\ 
2&$\Delta_0[\Xp{2}]^*
+\Delta_1[\Xq{2}]^*$
&${\displaystyle \Delta_0[\Xp{2}]^*
+\Delta_1[\Xpq{1}{};\Xp{}]^*}$
&$\Delta_0[\Xp{2}]^*$
\\[0.2 cm]
&
${\displaystyle 
+\frac{\Delta-\Delta_0-\Delta_1}2[\Freep 2]^*}$
&
${\displaystyle
+\frac{\Delta-\Delta_0-\Delta_1}2[\Freep 2]^*}$
&${\displaystyle +\frac{\Delta-\Delta_1}2[\Freep 2]^*}$
\\[0.4 cm]
\hline
 &&&\\ 
1&$\Delta_0\big[\Xpq{1}{};\Xq{}\big]^*
+\Delta_1[\Xq{2}]^*$
&${\displaystyle
\Delta_{\min}\big[\Xpq{1}{}\big]^*
+\frac{\Delta_0-\Delta_{\min}}2[\tSpqk{1}{0}{2}]^*}$
&$
\Delta_0[\tSpqk{1}{0}{2}]^*
$
\\[0.4 cm]
&
${\displaystyle +\frac{\Delta-\Delta_0-\Delta_1}2[\Freep 2]^*}$
&
${\displaystyle +\frac{\Delta_1-\Delta_{\min}}2[\tSpqk{0}{1}{2}]^*
+\frac{\Delta-\Delta_{\max}}2[\Freep 2]^*}$
&${\displaystyle +\frac{\Delta-\Delta_0}2[\Freep 2]^*}$
\\[0.4 cm]
\hline
 &&&\\ 
$\leq 0$&${\displaystyle
\Delta_1[\Xq{2}]^*
+\frac{\Delta-\Delta_1}2[\Freep 2]^*}$
&${\displaystyle
\frac{\Delta_1}2[\tSpqk{0}{1}{2}]^*+\frac{\Delta-\Delta_1}2[\Freep 2]^*}$
&${\displaystyle \frac{\Delta}2[\Freep 2]^*}
$
\\[0.4 cm]
\hline
\end{tabular}
\end{center}

\vspace{0.2 cm}
Here follows a geometric description of the equivariant spaces appearing in the table.

\vspace{0.2 cm}

\begin{description}\setlength{\itemsep}{7pt}
\item [$\Xp{2}$\,] Stands for a $\PP^1$ which is pointwise fixed by the $C_2 $-action, even at the affine level.

\item [$\Xq{2}$\,] Stands for a $\PP^1$ which is pointwise fixed by the $C_2 $-action, however the affine lines representing the projective points are flipped by the action.

\item [$\Xpq 1{}$\,] Stands for a $\PP^1$, which only has two fixed points. The two affine lines corresponding to these points are respectively pointwise fixed and flipped by the $C_2$-action.

\item [$\big(\Xpq{1}{};\Xq{ }\big)$\,] stands for a $\PP^1$, as in the previous entry, but with the fixed point $\Xq{}$ taken as the singular part. This modification allows us to reinterpret the dimension: The natural dimension of $\Xpq 1{}$ is $\omega+\chi\omega-2$,
but we consider it to have dimension $\omega+2\chi\omega -2-2\sigma = (\omega+\chi\omega-2)+(\chi\omega-2\sigma)$,
which is possible because $\chi\omega = 2\sigma$ away from the fixed point $\Xq{}$. This puts $[\Xpq{1}{};\Xq{ }]^*$ in the correct grading.

\item [$\big(\Xpq{1}{};\Xp{ }\big)$\,] stands for a $\PP^1$, as in the previous entry, but this time the fixed point $\Xp{}$ is taken as the singular part. Again, this modification allows us to reinterpret the dimension.

\item [$\tSpqk{1}{0}{2}$\,] Stands for the desingularization of $\Spqk{1}{0}{2}$. The latter is the union of two $\PP^1$s intersecting in a fixed point, which corresponds to an  affine line pointwise fixed by the action.

\item [$\tSpqk{0}{1}{2}$\,] Stands for the desingularization of $\Spqk{0}{1}{2}$. The latter is the union of two $\PP^1$s intersecting in a fixed point, which corresponds to an  affine line flipped by the $C_2$-action.

\item [$\Freep 2$] Stands for a pair of $\PP^1$'s that are exchanged by the $C_2$-action.
\end{description}

\subsection{Special case: dimension 2}

Giving the full list of all the different scenarios that can arise in dimension two would not add much to what we have already seen, so we have chosen to focus on two cases which highlight some aspects of Theorem~\ref{thm:Bezout} that cannot be appreciated in dimensions 0 and 1. 

We first consider the case in which $m=m_0=m_1=\ell=3$. This is only possible  if $n_\II=0$, which forces $d_\II=1$ and hence $\epsilon=1$. In this setting the formula yields
    \begin{align*}
        e(F)
        &= \rem{\binom{3}{1}} \Big[\Xpq {2}{}; \Xq{ }\union \Xp{2}\Big]^*+
    \rem{\binom{3}{2}} \Big[\Xpq {1}{2}; \Xq{2}\union \Xp{ }\Big]^* \\
        &\qquad + \Delta_0 [\Xp{3}]^* + \Delta_1 [\Xq{3}]^*  + \frac{\Delta - \Delta_0 - \Delta_1 - (2^{2}-2)}{2} [\Freep 3]^*\\
&= \Big[\Xpq {2}{}; \Xq{ }\union \Xp{2}\Big]^*+
     \Big[\Xpq {1}{2}; \Xq{2}\union \Xp{ }\Big]^* \\
        &\qquad + \Delta_0 [\Xp{3}]^* + \Delta_1 [\Xq{3}]^*  + \frac{\Delta - \Delta_0 - \Delta_1  -2}{2} [\Freep 3]^*
        \,.
    \end{align*}
 This is the first instance in which the binomial coefficients appearing in the formula are not multiples of 2.
 As in the dimension 1 case, the primary function of the singular parts is to allow us to reinterpret dimensions appropriately.
 
As another example, we consider the case in which $m=3$, $m_0=2$ and $m_1=1$. In this case the formula yields
   \begin{align*}
        e(F) &= \frac{\Delta_{\min}}{2}[\tSpqk{2}{1}3]^* 
         + \frac{\Delta_0-\Delta_{\min}}{2}[\tSpqk{2}{0}3; \tSpqk{1}{0}{2}]^* \\
        &\qquad + \frac{\Delta_1-\Delta_{\min}}{2}[\tSpqk{1}{1}3; \tSpqk{1}{0}{2}]^* 
        + \frac{\Delta - \Delta_{\max}}{2}[\Freep 3]^*\,.
    \end{align*}
 This is the simplest case in which the binate Schubert varieties appear as the singular part of a singular manifold. As an example, let's examine more closely the singular manifold $(\tSpqk{1}{1}3; \tSpqk{1}{0}{2}).$
Here $\tSpqk{1}{1}3$ is the desingularization of $\Spqk{1}{1}3$, the latter consisting of two $\PP^2$s exchanged by the $C_2$-action and intersecting in $\Xpq1{}$. On the other hand, the  singular part $\tSpqk{1}{0}{2}$ projects onto two $\PP^1$s intersecting in $\Xp{}$. 
The intersections in both cases are desingularized by crossing with $S(\R^{1+\sigma})$.

\appendix

\section{Equivariant ordinary cohomology}

\subsection{Ordinary cohomology}\label{app:ordinarycohomology}
In this paper
we use $\GG$-equivariant ordinary cohomology with the extended grading developed in \cite{CostenobleWanerBook}.
This is an extension of Bredon's ordinary cohomology to be graded on representations of the fundamental groupoids
of $\GG$-spaces.
We review here some of the notation and computations we use.
A more detailed description of this theory can also be found in \cite{CHTFiniteProjSpace}.

For an ex-$\GG$-space $Y$ over $X$, we write $H_\GG^{RO(\Pi X)}(Y;\Mackey T)$ for the
ordinary cohomology of $Y$ with coefficients in a Mackey functor $\Mackey T$, graded
on $RO(\Pi X)$, the representation ring of the fundamental groupoid of $X$.
Through most of this paper we use the Burnside ring Mackey functor $\Mackey A$ as the coefficients,
and write simply $H_\GG^{RO(\Pi X)}(Y)$.

In \cite{CostenobleWanerBook} and \cite{CHTFiniteProjSpace} we considered cohomology to be Mackey functor--valued,
which is useful for many computations, and wrote $\Mackey H_\GG^{RO(\Pi X)}(Y)$ for the resulting theory.
In this paper we concentrate on the values at level $\GG/\GG$, and write
$H_\GG^{RO(\Pi X)}(Y) = \Mackey H_\GG^{RO(\Pi X)}(Y)(\GG/\GG)$. However, we 
use extensively the structure maps of the Mackey functor structure, namely the
restriction functor $\rho$ from equivariant cohomology to nonequivariant cohomology, and the transfer
map $\tau$ going in the other direction.
We will also treat nonequivariant cohomology as graded on $RO(\Pi X)$ via the forgetful
map $RO(\Pi_\GG X)\to RO(\Pi_e X)$ from the representation ring of the equivariant fundamental groupoid
of $X$ to the representation ring of its nonequivariant fundamental groupoid.
Another way of saying that is that we view nonequivariant cohomology as
\[
    H^{RO(\Pi X)}(X;\ZZ) = \Mackey H^{RO(\Pi X)}_\GG(X;\Mackey A)(\GG/e).
\]

For all $X$ and $Y$, $H_\GG^{RO(\Pi X)}(Y)$ is a graded module over 
\[
    \copt = \copt^{RO(\GG)} = H_\GG^{RO(\GG)}(S^0),
\]
the cohomology of a point.
The grading on the latter is just $RO(\GG)$, the real representation ring of $\GG$, which is free abelian
on $1$, the class of the trivial representation $\R$, and $\sigma$, the class of the
sign representation $\R^\sigma$. The cohomology of a point was calculated by Stong in an unpublished
manuscript and first published by Lewis in \cite{LewisCP}. We can picture the calculation as 
in Figure~\ref{fig:cohompt},
in which a group in grading $a+b\sigma$ is plotted at the point $(a,b)$, and the spacing of the grid lines
is 2 (which is more conventient for other graphs we will give).
The square box at the origin is a copy of $A(\GG)$, the Burnside ring of $\GG$,
closed circles are copies of $\Z$, and open circles are copies of $\Z/2$.
\begin{figure}
\begin{tikzpicture}[x=4mm, y=4mm]
	\draw[step=2, gray, very thin] (-7.8, -7.8) grid (7.8, 7.8);
	\draw[thick] (-8, 0) -- (8, 0);
	\draw[thick] (0, -8) -- (0, 8);
    \node[right] at (8,0) {$a$};
    \node[above] at (0,8) {$b\sigma$};

    \fill (-0.3, -0.3) rectangle (0.3, 0.3);
    \fill (0, -7) circle(0.2);
    \fill (0, -6) circle(0.2);
    \fill (0, -5) circle(0.2);
    \fill (0, -4) circle(0.2);
    \fill (0, -3) circle(0.2);
    \fill (0, -2) circle(0.2);
    \fill (0, -1) circle(0.2);
    \fill (0, 1) circle(0.2);
    \fill (0, 2) circle(0.2);
    \fill (0, 3) circle(0.2);
    \fill (0, 4) circle(0.2);
    \fill (0, 5) circle(0.2);
    \fill (0, 6) circle(0.2);
    \fill (0, 7) circle(0.2);

    \fill (-2, 2) circle(0.2);
    \fill (-4, 4) circle(0.2);
    \fill (-6, 6) circle(0.2);

    \draw[fill=white] (-2, 3) circle(0.2);
    \draw[fill=white] (-2, 4) circle(0.2);
    \draw[fill=white] (-2, 5) circle(0.2);
    \draw[fill=white] (-2, 6) circle(0.2);
    \draw[fill=white] (-2, 7) circle(0.2);
    \draw[fill=white] (-4, 5) circle(0.2);
    \draw[fill=white] (-4, 6) circle(0.2);
    \draw[fill=white] (-4, 7) circle(0.2);
    \draw[fill=white] (-6, 7) circle(0.2);

    \fill (2, -2) circle (0.2);
    \fill (4, -4) circle (0.2);
    \fill (6, -6) circle (0.2);
    \draw[fill=white] (3, -3) circle(0.2);
    \draw[fill=white] (5, -5) circle(0.2);
    \draw[fill=white] (7, -7) circle(0.2);
    
    \draw[fill=white] (3, -4) circle(0.2);
    \draw[fill=white] (3, -5) circle(0.2);
    \draw[fill=white] (3, -6) circle(0.2);
    \draw[fill=white] (3, -7) circle(0.2);
    \draw[fill=white] (5, -6) circle(0.2);
    \draw[fill=white] (5, -7) circle(0.2);

    \node[right] at (0,1) {$e$};
    \node[below left] at (-2,2) {$\xi$};
    \node[above right] at (2,-2) {$\tau(\iota^{-2})$};
    \node[left] at (0,-1) {$e^{-1}\kappa$};

\end{tikzpicture}
\caption{$\copt^{RO(\GG)}$}\label{fig:cohompt}
\end{figure}

Recall that $A(\GG)$ is the Grothendieck group of finite $\GG$-sets, with multiplication given by products of sets.
Additively, it is free abelian on the classes of the orbits of $\GG$, for which we write
$1 = [\GG/\GG]$ and $g = [\GG/e]$. The multiplication is given by $g^2 = 2g$.
We also write $\kappa = 2 - g$. Other important elements are shown in the figure:
The group in degree $\sigma$ is generated by an element $e$,
while the group in degree $-2 + 2\sigma$ is generated by an element $\xi$.
The groups in the second quadrant are generated by the products $e^m\xi^n$, with $2e\xi = 0$.
We have $g\xi = 2\xi$ and $ge = 0$.
The groups in gradings $-m\sigma$, $m\geq 1$, are generated by elements $e^{-m}\kappa$, so named
because $e^m\cdot e^{-m}\kappa = \kappa$. We also have $ge^{-m}\kappa = 0$.

To explain $\tau(\iota^{-2})$, we think for moment about the nonequivariant cohomology
of a point. If we grade it on $RO(\GG)$, we get
$H^{RO(\GG)}(S^0;\Z) \iso \Z[\iota^{\pm 1}]$, where $\deg \iota = -1 + \sigma$.
(Nonequivariantly, we cannot tell the difference between $\R$ and $\R^\sigma$.)
We have $\rho(\xi) = \iota^2$ and $\tau(\iota^2) = g\xi = 2\xi$.
Note also that $\tau(1) = g$.
In the fourth quadrant we have that the group in grading $n(1-\sigma)$, $n\geq 2$, is
generated by $\tau(\iota^{-n})$.
The remaining groups in the fourth quadrant will not concern us here.
For more details, see \cite{Co:InfinitePublished} or \cite{CHTFiniteProjSpace}.

\subsection{The cohomology of projective space}\label{app:cohomologyprojective}
We summarize the cohomology of $\Xpq pq$ as calculated in \cite{CHTFiniteProjSpace}.

Write $B = \Xpq \infty\infty$.
Its fixed set is
\[
    B^\GG = \Xp \infty \disjunion \Xq \infty = B^0 \disjunion B^1,
\]
where we use the indices 0 and 1 to evoke the trivial and nontrivial representation
of $\GG$, respectively.
(We use this convention throughout, that an index 0 refers to something related to
$B^0$ and an index 1 refers to something related to $B^1$.)
Representations of $\Pi B$ are determined by their restrictions to $B^0$ and $B^1$,
which are elements of $RO(\GG)$ that must have the same nonequivariant rank and
the same parity for the ranks of their fixed point representations.
As shown in \cite[\S2.2, p.13]{CHTFiniteProjSpace}, this leads to the calculation
that $RO(\Pi B)$ is free abelian on three generators, but we find
the following presentation more useful:
\[
    RO(\Pi B) = \Z\{1,\sigma,\Omega_0,\Omega_1\}/\langle \Omega_0 + \Omega_1 = 2\sigma - 2\rangle,
\]
where $\Omega_0$ is the representation whose value on $B^0$ is $2\sigma - 2$  and
on $B^1$ is $0$;
while $\Omega_1$ is the representation whose value on $B^0$ is $0$ and
on $B^1$ is $2\sigma-2$.
For any $\alpha\in RO(\Pi B)$, write $|\alpha|\in\ZZ$ for its underlying nonequivariant rank,
and $\alpha_0$ and $\alpha_1 \in RO(\GG)$ for its retrictions to $B^0$ and $B^1$, respectively.
What we said above can be phrased as: $\alpha$ is completely determined by the triple of ranks
$(|\alpha|,|\alpha_0^\GG|,|\alpha_1^\GG|)$, where the latter two ranks have the same parity.

We think of the finite projective spaces as spaces over $B$ by the evident inclusions
$\Xpq pq \to \Xpq\infty\infty$, so will grade their cohomologies on $RO(\Pi B)$.
Let $\omega$ denote the tautological line bundle over $B$, let $\omega\dual$ be its dual bundle, let
$\chi\omega = \omega\tensor_\C \C^\sigma$, and let $\chi\omega\dual$ be the dual of $\chi\omega$.
We also use the notation from algebraic geometry in which $\omega = O(-1)$ and $\omega\dual = O(1)$;
we write $\chi O(-1) = \chi\omega$ and $\chi O(1) = \chi\omega\dual$.

Associated to any bundle over $B$ is a representation in $RO(\Pi B)$ that we think of as the
equivariant rank of the bundle; this representation is given by the fiber representations
over $B^0$ and $B^1$. In the case of $\omega$ and $\chi\omega$, we have
\begin{align*}
    \omega &= 2 + \Omega_1 \\
    \chi\omega &= 2 + \Omega_0,
\end{align*}
where we write $\omega$ and $\chi\omega$ again for the associated elements of $RO(\Pi B)$.

Let $\cwd$ and $\cxwd$ denote the Euler classes of $\omega\dual$ and $\chi\omega\dual$, respectively.
The cohomology of $\Xpq\infty\infty$ was calculated in \cite{Co:InfinitePublished} as follows.

\begin{theorem}[{\cite[Theorem 11.3]{Co:InfinitePublished}}]
$H_{\GG}^{RO(\Pi B)}(B_+)$ is an algebra over $\copt$ generated by the 
Euler classes $\cwd$ and $\cxwd$ together with classes $\cxwt$ and $\cwt$. 
These elements live in gradings
\begin{alignat*}{4}
 \grad\cwd &= \omega  \qquad&  \grad\cxwd &= \chiw  \\
 \grad\cwt &= \omega - 2  \qquad& \grad\cxwt &= \chiw-2 
\end{alignat*}
They satisfy the relations
\begin{align*}
		\cxwt \cwt &= \xi \\
        \mathllap{\text{and}\quad}\cwt \cxwd &= (1-\kappa)\cxwt \cwd + e^2,
\end{align*}
and these relations completely determine the algebra.
Moreover, $H_{\GG}^{RO(\Pi B)}(B_+)$ is free as a module over $\copt$.
\qed
\end{theorem}

There are two restriction maps we will use,
\begin{align*}
    \rho\colon H_\GG^{\alpha}(B_+) &\to H^{\alpha}(B_+),
\intertext{restriction to nonequivariant cohomology, and}
    (-)^\GG\colon H_\GG^\alpha(B_+) &\to H^{\alpha_0^\GG}(B^0_+) \dirsum H^{\alpha_1^\GG}(B^1_+),
\end{align*}
the fixed-point map. 
In the case of $\rho$, recall that we are considering nonequivariant cohomology to
be graded on $RO(\Pi B)$ via the forgetful homomorphism
$RO(\Pi_\GG B)\to RO(\Pi B) \iso \ZZ$, which is just the map $\alpha\mapsto |\alpha|$.
When we impose this regrading, $H^{RO(\Pi B)}(B_+)$ acquires two new invertible elements,
\[
    \iota \in H^{\sigma-1}(B_+) \qquad\text{and}\qquad \zeta\in H^{\omega-2}(B_+)
\]
expressing the fact that $|\sigma - 1| = 0 = |\omega-2|$.

The restriction maps are ring maps and their values on the multiplicative generators are
\begin{equation}\label{eqn:restrictions}
\begin{aligned}
    \rho(\cxwt) &= \iota^2\zeta^{-1} &  \cxwt^\GG &= (0,1) \\
    \rho(\cwt) &= \zeta & \cwt^\GG &= (1,0) \\
    \rho(\cwd) &= \zeta\cd & \cwd^{\;\GG} &= (\cd, 1) \\
    \rho(\cxwd) &= \iota^2\zeta^{-1}\cd \qquad\qquad& \cxwd^{\;\GG} &= (1, \cd).
\end{aligned}
\end{equation}
Here, $\cd \in H^2(\Xp{\infty}_+)$ denotes the first nonequivariant Chern class of $O(1)$.
There are similar restriction maps
\begin{align*}
    \rho\colon \copt^\alpha &\to H^{\alpha}(S^0), \\
    (-)^\GG\colon \copt^\alpha &\to H^{\alpha^\GG}(S^0).
\end{align*}
The values on the most interesting elements are
\begin{equation}\label{eqn:restrpoint}
\begin{aligned}
    \rho(\xi^k) &= \iota^{2k} & \rho(\tau(\iota^{2k})) &= 2\iota^{2k} & \rho(e^{-k}\kappa) &= 0 & \rho(e^k) &= 0 \\
    (\xi^k)^\GG &= 0 & \tau(\iota^{2k})^\GG &= 0 & (e^{-k}\kappa)^\GG &= 2 & (e^k)^\GG &= 1.
\end{aligned}
\end{equation}

We also have the Frobenius relation
\begin{equation}\label{eqn:Frobenius}
    b\tau(a)=\tau(\rho(b)a),
\end{equation}
which is very useful for calculations involving $\tau$.

Moving now to finite projective spaces,
on pulling back along the inclusion $\Xpq pq \includesin \Xpq\infty\infty$,
the cohomology of $\Xpq pq$ contains elements
we will also call $\cwd$, $\cxwd$, $\cxwt$, and $\cwt$.
In \cite{CHTFiniteProjSpace}, we showed the following.

\begin{theorem}[{\cite[Theorem A]{CHTFiniteProjSpace}}]\label{thm:cohomStructure}
Let $0 \leq p < \infty$ and $0 \leq q < \infty$ with $p+q > 0$.
Then $H_{\GG}^{RO(\Pi B)}(\Xpq{p}{q}_+)$ is a free module over $\copt$.
As a (graded) commutative algebra over $\copt$, the ring
$H_{\GG}^{RO(\Pi B)}(\Xpq{p}{q}_+)$ is generated by $\cwd$, $\cxwd$, $\cxwt$, and $\cwt$,
together with the following classes:
$\cwd^{\;p}$ is infinitely divisible by $\cxwt$, meaning that, for $k\geq 1$,
there are unique elements $\cxwt^{-k}\cwd^{\;p}$ such that
\[
 \cxwt^k \cdot \cxwt^{-k} \cwd^{\;p} = \cwd^{\;p}.
\]
Similarly, $\cxwd^{\;q}$ is infinitely divisible by $\cwt$, meaning that, for $k\geq 1$,
there are unique elements $\cwt^{-k}\cxwd^{\;q}$ such that
\[
 \cwt^k \cdot \cwt^{-k} \cxwd^{\;q} = \cxwd^{\;q}.
\]
The generators satisfy the following further relations:
\begin{align*}
	\cxwt \cwt &= \xi, \\
	\cwt \cxwd &= (1-\kappa)\cxwt \cwd + e^2 \\
    \mathllap{\text{and}\qquad} \cwd^{\;p} \cxwd^{\;q} &= 0.
\end{align*}
\qed
\end{theorem}

We also
gave an explicit basis for $H_\GG^{RO(\Pi B)}(\Xpq pq_+)$ over $\copt$,
which we can describe as follows. 
We define sets $F_{p,q}(m)$, recursively on $p$ and $q$, that give bases for $H_\GG^{m\omega+RO(\GG)}(\Xpq pq_+)$.
(The superscript indicates the sub-$\copt$-module of groups graded on the coset
$m\omega+RO(\GG)\subseteq RO(\Pi B)$.)
For $m\in\Z$, let
\begin{align*}
 F_{p,0}(m) &:= \{ \cwt^m, \cwt^{m-1}\cwd, \cwt^{m-2}\cwd^{\;2},\ldots, \cwt^{m-p+1}\cwd^{\;p-1} \} \\
\intertext{and}
 F_{0,q}(m) &:= \{ \cxwt^m, \cxwt^{m-1}\cxwd, \cxwt^{m-2}\cxwd^{\;2}, \ldots, \cxwt^{m-q+1}\cxwd^{\;q-1} \}.
\end{align*}
(Note that $\cwt$ is invertible in the first case and $\cxwt$ is invertible in the second.)
For $p, q > 0$ we then define
\[
 F_{p,q}(m) := 
  \begin{cases}
    \{ \cwt^m \} \union i_! F_{p-1,q}(m-1) & \text{if\ \,$m\geq 0$} \\
    \{ \cxwt^{|m|} \} \union j_! F_{p,q-1}(m+1) & \text{if\,\ $m < 0$,}
  \end{cases}
\]
where $i\colon \Xpq{p-1}q \to \Xpq pq$ and $j\colon \Xpq p{(q-1)} \to \Xpq pq$ are the inclusions.
The pushforward $i_!$ is given algebraically by multiplication by $\cwd$ and $j_!$ is multiplication by $\cxwd$.

It is possible from this description to write down the bases explicitly, but the results are messy, having to be broken
down by cases depending on where $m$ falls in relation to $p$ and $q$;
this is done in \cite[Proposition~4.7]{CHTFiniteProjSpace}.
However, we can make the following general statements.
\begin{enumerate}
    \item For fixed $p$, $q$, and $m$, there are exactly $p+q$ basis elements lying in $H_\GG^{m\omega+RO(\GG)}(\Xpq pq_+)$.
    \item Those basis elements have gradings of the form $m(\omega-2) + 2a_i + 2b_i\sigma$, $0\leq i \leq p+q-1$, where
            $a_i + b_i = i$.
    \item The basis element with grading $m(\omega-2) + 2a + 2b\sigma$ restricts to the nonequivariant class $\widehat c^{\;a+b}$, where,
            again, $\widehat c$ is the first nonequivariant Chern class of $O(1)$.
    \item For a given integer $k$, there are at most two indices $i$ such that $a_i = k$.
\end{enumerate}
Figure~\ref{fig:bases} illustrates, in the case of $\Xpq 45$, how the basis elements are arranged for various values of $m$.
In each case, the basis element with grading $m(\omega-2) + 2a + 2b\sigma$ is marked by a dot at coordinates $(a,b)$.
\begin{figure}
\[
\begin{tikzpicture}[scale=0.4]
	\draw[step=1cm, gray, very thin] (-7.8, -0.8) grid (5.8, 7.8);
	\draw[thick] (-8, 0) -- (6, 0);
	\draw[thick] (0, -1) -- (0, 8);

 	\node[below] at (-1, -1) {$m=-6$};

    \fill (-6, 6) circle(5pt);
    \fill (-5, 6) circle(5pt);
    \fill (-4, 6) circle(5pt);
    \fill (-3, 6) circle(5pt);
    \fill (-2, 6) circle(5pt);

    \fill (0, 5) circle(5pt);
    \fill (1, 5) circle(5pt);
    \fill (2, 5) circle(5pt);
    \fill (3, 5) circle(5pt);

\end{tikzpicture}
\hspace{0.5 cm}
\begin{tikzpicture}[scale=0.4]
	\draw[step=1cm, gray, very thin] (-6.8, -0.8) grid (6.8, 7.8);
	\draw[thick] (-7, 0) -- (7, 0);
	\draw[thick] (0, -1) -- (0, 8);

 	\node[below] at (0, -1) {$m=-3$};

    \fill (-3, 3) circle(5pt);
    \fill (-2, 3) circle(5pt);
    \fill (-1, 3) circle(5pt);
    \fill ( 0, 3) circle(5pt);
    \fill ( 0, 4) circle(5pt);

    \fill (1, 4) circle(5pt);
    \fill (1, 5) circle(5pt);
    \fill (2, 5) circle(5pt);
    \fill (3, 5) circle(5pt);

\end{tikzpicture}
\]
\[
\begin{tikzpicture}[scale=0.4]
	\draw[step=1cm, gray, very thin] (-6.8, -0.8) grid (6.8, 5.8);
	\draw[thick] (-7, 0) -- (7, 0);
	\draw[thick] (0, -1) -- (0, 6);

 	\node[below] at (0, -1) {$m=0$};

    \fill (0, 0) circle(5pt);
    \fill (1, 1) circle(5pt);
    \fill (2, 2) circle(5pt);
    \fill (3, 3) circle(5pt);

    \fill (0, 1) circle(5pt);
    \fill (1, 2) circle(5pt);
    \fill (2, 3) circle(5pt);
    \fill (3, 4) circle(5pt);
    \fill (4, 4) circle(5pt);

\end{tikzpicture}
\hspace{0.5 cm}
\begin{tikzpicture}[scale=0.4]
	\draw[step=1cm, gray, very thin] (-5.8, -0.8) grid (7.8, 5.8);
	\draw[thick] (-6, 0) -- (8, 0);
	\draw[thick] (0, -1) -- (0, 6);

 	\node[below] at (1, -1) {$m=2$};

    \fill (0, 0) circle(5pt);
    \fill (1, 0) circle(5pt);
    \fill (2, 1) circle(5pt);
    \fill (3, 2) circle(5pt);

    \fill (2, 0) circle(5pt);
    \fill (3, 1) circle(5pt);
    \fill (4, 2) circle(5pt);
    \fill (5, 2) circle(5pt);
    \fill (6, 2) circle(5pt);

\end{tikzpicture}
\]
\[
\begin{tikzpicture}[scale=0.4]
	\draw[step=1cm, gray, very thin] (-1.8, -3.8) grid (11.8, 2.8);
	\draw[thick] (-2, 0) -- (12, 0);
	\draw[thick] (0, -4) -- (0, 3);

 	\node[below] at (5, -4) {$m=6$};

    \fill (6, -2) circle(5pt);
    \fill (7, -2) circle(5pt);
    \fill (8, -2) circle(5pt);
    \fill (9, -2) circle(5pt);
    \fill (10, -2) circle(5pt);

    \fill (0, 0) circle(5pt);
    \fill (1, 0) circle(5pt);
    \fill (2, 0) circle(5pt);
    \fill (3, 0) circle(5pt);

\end{tikzpicture}
\]
\caption{Bases of $H_\GG^{m\omega+RO(\GG)}(\Xpq 45_+)$}
\label{fig:bases}
\end{figure}

\clearpage

\bibliography{Bibliography}{}
\bibliographystyle{amsplain} 

\end{document}